\documentclass{article}

\usepackage[utf8]{inputenc} 
\usepackage{amssymb}
\usepackage{polski}
\usepackage[english]{babel}
\usepackage{amsmath}
\usepackage{amsthm}
\usepackage{bbm}
\usepackage{dsfont}
\usepackage{graphicx}
\usepackage[colorinlistoftodos]{todonotes}
\usepackage{hyperref}
\usepackage{enumerate}
\usepackage{systeme}
\usepackage{bm}
\usepackage[capitalise]{cleveref}
\usepackage{mathtools}
\usepackage{todonotes}
\usepackage{stmaryrd}
\usepackage[shortlabels]{enumitem}
\setlist[description]{labelindent=\parindent,parsep=2mm,topsep=10pt,labelwidth=12mm,leftmargin=\parindent+\labelwidth+2mm}

\newtheorem{theorem}{Theorem}[section]
\newtheorem{proposition}[theorem]{Proposition}
\newtheorem{lemma}[theorem]{Lemma}
\newtheorem{corollary}[theorem]{Corollary}

\theoremstyle{definition}
\newtheorem{definition}[theorem]{Definition}

\newtheorem{question}[theorem]{Question}

\newtheorem{claim}{Claim}[theorem]

\crefname{theorem}{Theorem}{Theorems}
\Crefname{theorem}{Theorem}{Theorems}

\crefname{proposition}{Proposition}{Propositions}
\Crefname{proposition}{Proposition}{Propositions}

\crefname{lemma}{Lemma}{Lemmas}
\Crefname{lemma}{Lemma}{Lemmas}

\crefname{definition}{Definition}{Definitions}
\Crefname{definition}{Definition}{Definitions}

\crefname{corollary}{Corollary}{Corollaries}
\Crefname{corollary}{Corollary}{Corollaries}

\crefname{remark}{Remark}{Remarks}
\Crefname{remark}{Remark}{Remarks}

\crefname{question}{Question}{Questions}
\Crefname{question}{Question}{Questions}

\crefname{claim}{Claim}{Claims}
\Crefname{claim}{Claim}{Claims}

\crefname{fact}{Fact}{Facts}
\Crefname{fact}{Fact}{Facts}

\makeatletter
\@addtoreset{claim}{lemma}
\makeatother

\renewcommand{\theclaim}{\arabic{claim}}

\newenvironment{claimproof}{
  \par\noindent\textit{Proof of Claim \theclaim.} 
}{
  \hfill $\dashv$ \par
}

\newcommand{\rca}{\mathsf{RCA}_0}
\newcommand{\rcas}{\mathsf{RCA}_0^*}
\newcommand{\wkl}{\mathsf{WKL}_0}
\newcommand{\wkls}{\mathsf{WKL}_0^*}
\newcommand{\aca}{\mathsf{ACA}_0}
\newcommand{\cac}{\mathsf{CAC}}
\newcommand{\ads}{\mathsf{ADS}}

\newcommand{\crt}{\mathsf{CRT}^2_2}

\newcommand{\rt}{\mathsf{RT}^2_2}

\newcommand{\rtnk}{\mathsf{RT}^n_k}

\newcommand{\coh}{\mathsf{COH}}

\newcommand{\expax}{\mathsf{exp}}

\newcommand{\pra}{\mathsf{PRA}}
\newcommand{\pa}{\mathsf{PA}}
\newcommand{\pam}{\mathsf{PA}^-}

\newcommand{\sca}{\mathsf{SC}}
\newcommand{\sctheory}{\mathsf{I}\Delta_0+\mathsf{exp} + \mathsf{SC}}
\newcommand{\lpc}{\sigzo\textrm{-}\mathsf{LPC}}

\newcommand{\bszo}{\mathsf{B}\Sigma^0_1}
\newcommand{\iszo}{\mathsf{I}\Sigma^0_1}

\newcommand{\idzz}{\mathsf{I}\Delta^0_0}

\newcommand{\iso}{\mathsf{I}\Sigma_1}

\newcommand{\idz}{\mathsf{I}\Delta_0}

\newcommand{\sigzo}{\ensuremath{\Sigma^0_1} }
\newcommand{\pizo}{\ensuremath{\Pi^0_1} }
\newcommand{\delzo}{\ensuremath{\Delta^0_1} }

\newcommand{\delzz}{\ensuremath{\Delta^0_0} }

\newcommand{\pioo}{\ensuremath{\Pi^1_1} }

\newcommand{\sigzt}{\ensuremath{\Sigma^0_2} }
\newcommand{\sigzth}{\ensuremath{\Sigma^0_3} }
\newcommand{\pizt}{\ensuremath{\Pi^0_2} }
\newcommand{\pizth}{\ensuremath{\Pi^0_3} }

\newcommand{\delz}{\ensuremath{\Delta_0}}

\newcommand{\sigo}{\ensuremath{\Sigma_1}}
\newcommand{\pio}{\ensuremath{\Pi_1}}
\newcommand{\delo}{\ensuremath{\Delta_1}}

\newcommand{\sigzn}{\ensuremath{\Sigma_n^0} }
\newcommand{\pizn}{\ensuremath{\Pi_n^0} }
\newcommand{\delzn}{\ensuremath{\Delta_n^0} }
\newcommand{\sign}{\ensuremath{\Sigma_n}}
\newcommand{\pin}{\ensuremath{\Pi_n}}
\newcommand{\deln}{\ensuremath{\Delta_n}}

\newcommand{\ele}[2]{\exists #1 \! \le \! #2 \,}
\newcommand{\fale}[2]{\forall #1 \! \le \! #2 \,}
\newcommand{\el}[2]{\exists #1 \! < \! #2 \,}
\newcommand{\fal}[2]{\forall #1 \! < \! #2 \,}

\newcommand{\fage}[2]{\forall #1 \! \geq \! #2 \,}
\newcommand{\eg}[2]{\exists #1 \! > \! #2 \,}

\newcommand{\lone}{\mathcal{L}_\mathrm{I}}
\newcommand{\ltwo}{\mathcal{L}_\mathrm{II}}

\newcommand{\nn}{\mathbb{N}}
\newcommand{\cutzo}{\mathrm{I}^0_1}

\newcommand{\codmi}{\mathrm{Cod}(M/I)}

\newcommand{\Con}{\mathrm{Con}}
\newcommand{\Ack}{\mathrm{Ack}}

\newcommand{\vv}{\overline{v}}
\newcommand{\VV}{\overline{V}}
\newcommand{\xx}{\overline{x}}
\newcommand{\XX}{\overline{X}}
\newcommand{\zz}{\overline{z}}
\newcommand{\ZZ}{\overline{Z}}
\newcommand{\ww}{\overline{w}}
\newcommand{\yy}{\overline{y}}
\newcommand{\uu}{\overline{u}}

\newcommand{\X}{\mathcal{X}}

\newcommand{\restr}[2]{{#1}\!\restriction\!{#2}}

\newcommand{\defeq}{\!\vcentcolon\,=}

\newcommand{\cond}[2]{{#1}\!\in\!\mathrm{Cond}_{#2}}
\newcommand{\ncond}[2]{{#1}\!\notin\!\mathrm{Cond}_{#2}}
\newcommand{\lcond}[3]{{#1}\!\trianglelefteqslant_{#3}\!{#2}}

\newcommand{\nam}[2]{{#1}\!\in\!\mathrm{Name}_{#2}}
\newcommand{\condless}[1]{\!\trianglelefteqslant_{#1}\!}
\newcommand{\validname}[3]{{#1}\Vdash_{#2}\!{#3}\!\downarrow}

\newcommand{\forces}[3]{{#1}\Vdash_{#2}\!{#3}}
\newcommand{\nforces}[3]{{#1}\nVdash_{#2}\!{#3}}

\title{A non-speedup result for the chain-antichain principle over a weak base theory}
\author{Katarzyna W. Kowalik\thanks{Faculty of Mathematics, Informatics and Mechanics, University of Warsaw, Banacha 2, 02-097 Warszawa, Poland, 
\texttt{katarzyna.kowalik@mimuw.edu.pl}.}}

\date{}

\begin{document} 
\maketitle

\begin{abstract}
We show that the theory $\wkls+\cac$ is polynomially simulated by $\rcas$ with respect to $\forall\pizth$ formulas.   
For the proof, we use the method of forcing interpretations and syntactically simulate a two-step model-theoretic argument, which involves construction of a restricted definable ultrapower, followed by a generic cut satisfying $\cac$.
Our result sharply contrasts with the previously known fact that $\rcas+\rt$ has non-elementary speedup over $\rcas$. 
\end{abstract}

\section{Introduction}

The logical strength of Ramsey-like statements is arguably one of the richest and most active areas of research in reverse mathematics. In recent years Ramsey's theorem and some of its combinatorial consequences were studied over the weak base theory $\rcas$ \cite{weak_cousins, fkwy, kky:ramsey-rca0star, pfsize, mengzhou-COH}, which was introduced by Simpson and Smith in \cite{Simpson-Smith} and differs from the usually considered system $\rca$ in that it allows induction only for $\delzo$ formulas rather for $\sigzo$ formulas.
The weaker base theory makes it possible to calibrate the logical strength of mathematical theorems that are provable in $\rca$ and to track uses of $\sigzo$-induction in mathematical proofs. 
On the other hand, weakening induction axioms leads to some conceptual challenges that do not occur over $\rca$. 
In particular, the notion of an infinite set becomes less robust, as it is consistent with $\rcas$ 
that there exists an unbounded subset of $\nn$ which is not in bijective correspondence with all of $\nn$. 
Thus, when studying over $\rcas$ the strength of a mathematical theorem concerning infinite sets, one has to be precise and consistent about what one means by saying that a set $A\subseteq\nn$ is infinite.
Usually, one requires $A$ to be merely unbounded, i.e., to contain arbitrarily large elements. Another natural possibility is to require a bijection between $\nn$ and $A$.

In \cite{weak_cousins} and \cite{kky:ramsey-rca0star}, Ramsey's theorem for $n$-tuples and $k$ colours $\rtnk$, the chain-antichain principle $\cac$, the ascending-descending sequence principle $\ads$ and cohesive Ramsey's theorem $\crt$ were shown to be $\forall\pizth$- but not arithmetically conservative over $\rcas$. These results were obtained by some general techniques relying on the fact that all these principles exhibit certain distinctive model-theoretic behaviour: in the absence of $\iszo$ each of them is equivalent to its own relativization to a $\sigzo$-definable proper cut. 

However, it turns out that the above Ramsey-theoretic statements form a much more diverse group of principles than it might initially seem.
One way to see the differences is to look at their `long' versions, i.e. formulations requiring a solution set (i.e. a homogeneous set, chain or antichain etc.) to be in bijective correspondence with $\nn$. 
In \cite{weak_cousins} it is shown that the long versions of $\rtnk$ for $n, k\geq 2$ and of $\cac$, and one formulation of long $\ads$, imply $\rca$, whereas another possible formulation of long $\ads$ and the long version of $\crt$ are $\forall\pizth$-conservative over $\rcas$.

In this paper we adopt a more quantitative method of comparing axiomatic theories, namely proof size: instead of asking what is provable from a given set of axioms, we ask how feasibly it is provable.
Comparing proofs is particularly relevant when considering a conservation result of the form `$T_1$ is $\Gamma$-conservative over $T_2$', where $\Gamma$ is a class of sentences in the common language of $T_1$ and $T_2$. In such a case it is very natural to ask whether $T_1$ is able to provide essentially shorter proofs of those sentences than $T_2$.
It is an interesting phenomenon that one usually obtains one of just two contrasting answers: either the increase in size of proofs of sentences from $\Gamma$ is at most polynomial in $T_2$ with respect to $T_1$, or it cannot be bounded by any elementary computable function. In the latter case one says that $T_1$ has \textit{(non-elementary) speedup over $T_2$ with respect to $\Gamma$}. Here classical examples are $\aca$ over $\pa$ \cite{Pudlak_LC} and $\iso$ over $\pra$ \cite{ignjatovic_PhD}.

A result of the former kind is typically obtained by showing that $T_1$ is actually \textit{polynomially simulated by $T_2$ with respect to $\Gamma$}, i.e., that there exists a polynomial-time algorithm which, given as input a proof in $T_1$ of a sentence $\gamma\in\Gamma$, outputs a proof of $\gamma$ in $T_2$. Well-known conservation results that were strengthened to polynomial simulation are, for instance, arithmetical conservativity of $\rca$ over $\iso$
\cite{ignjatovic_PhD} and $\pioo$-conservativity of $\wkl$ over $\rca$ \cite{Avigad_paper, hajek_interpretability}.

In \cite{pfsize} Ko\l{}odziejczyk, Wong and Yokoyama studied $\rt$ with respect to proof size and showed that its behaviour depends on the base theory.

\begin{theorem}[{\cite[Theorem 3.1]{pfsize}}]\label{rt-speedup}
$\rcas+\rt$ has non-elementary speedup over $\rcas$ with respect to $\sigo$ sentences.  
\end{theorem}

\begin{theorem}[{\cite[Theorem 2.1]{pfsize}}]\label{rt-polynomial simulation}
$\wkl+\rt$ is polynomially simulated by $\rca$ with respect to $\forall\pizth$ sentences.    
\end{theorem}

The proof of \cref{rt-speedup} relies heavily on the classical exponential lower bound on Ramsey numbers for the finite version of $\rt$: there exists a $2$-colouring of $[2^{\frac{k}{2}}]^2$ without a homogeneous subset containing $k$ elements. This fact is used in \cite{pfsize} to prove a key lemma saying that $\rcas+\rt$ implies that the definable cut $\cutzo$ (which will also play a role here and will be defined later) is closed under exponentiation.

In this paper we focus on $\cac$, the strongest consequence of $\rt$ over $\rcas$ considered in \cite{weak_cousins}:
\begin{description}
\item[$\cac$] \itshape  For every partial order $(\nn, \preccurlyeq)$ there exists an unbounded set $S \subseteq \nn$ which is either a chain or an antichain in $\preccurlyeq$.
\end{description}
Here $S\subseteq \nn$ is a \textit{chain} in $\preccurlyeq$ if for all $x, y\in S$ either $x\preccurlyeq y$ or $x\succcurlyeq y$, and it is an \textit{antichain} in $\preccurlyeq$ if for all distinct $x, y\in S$ neither $x\preccurlyeq y$ nor $x\succcurlyeq y$.  
The starting point for our analysis of $\cac$ is the observation that the upper bound on Ramsey numbers corresponding to the finite version of $\cac$ is only polynomial: in every partial order on a set of size $k(k-1)$ there exists a chain or an antichain of size $k$. This is an easy consequence of the following classical result of Dilworth \cite{Dilworth}.

\begin{theorem}[Dilworth]\label{thm:Dilworth}
In every finite partial order $(P, \leq)$ the size of the largest antichain is equal to the smallest number $n$ such that $(P, \leq)$ is a union of $n$ chains.
\end{theorem}
This upper bound for finite $\cac$ can be used to show that $\cac$, in contrast to $\rt$, does not imply over $\rcas$ the closure of the cut $\cutzo$ under any super-polynomially growing function \cite[Theorem 3.16]{weak_cousins}. Thus, the argument from \cite{pfsize} used to prove \cref{rt-speedup} will not go through for $\cac$. On the other hand, the polynomial simulation result for $\wkl+\rt$ was obtained in \cite{pfsize} using the fact that an upper bound on finite Ramsey's theorem expressed in terms of so-called $\alpha$-large sets is polynomial. 
All this suggests that the polynomial upper bound for finite $\cac$ can be used to strengthen $\forall\pizth$-conservativity of $\wkls+\cac$ over $\rcas$ to polynomial simulation. We confirm this expectation and prove the following main result of this paper.

\begin{theorem}\label{main theorem_cac}
$\wkls+\cac$ is polynomially simulated by $\rcas$ with respect to $\forall\pizth$ sentences. 
\end{theorem}
\noindent 
As we will see, strengthening the theory $\rcas+\cac$ to $\wkls+\cac$ is very natural and requires no additional effort, both for the conservation result and for the polynomial simulation.
The proof of \cref{main theorem_cac} is obtained by transforming a two-step model-theoretic construction that can be used to prove (in an alternative way to that from \cite{kky:ramsey-rca0star}) $\forall\pizth$-conservativity of $\wkls+\cac$ over $\rcas$ into a syntactical forcing argument. 
The first step of the construction is a generic restricted ultrapower of a model of $\rcas$. Then, inside the ultrapower, one builds a generic proper cut satisfying $\cac$, using bounds on the finite version of $\cac$.
For each of these steps we will define a forcing interpretation, which can be seen as a syntactic simulation of the model-theoretic construction of a generic object.

The structure of the paper is as follows. 
After a preliminary Section 2, we describe the method of forcing interpretations in Section 3. Our main theorem is proved in Section 4, which consists of three parts. In the first two we define forcing interpretations corresponding to, respectively, the generic ultrapower and the generic cut. In the third part we complete the proof by composing the two forcing interpretations.

The paper is based on Chapter 4 of the author's PhD thesis \cite{kasia:rozprawa}, which provides additional details and motivation for the main result.


\section{Preliminaries}\label{chapter:preliminaries}
Standard monographs on first- and second-order arithmetic are \cite{HP_Book} and \cite{Dzhafarov_Mummert_Book, Simpson_Book}, respectively. For an introduction to the topic of proof size we refer the reader to the survey by Pudl\'ak \cite{Pudlak_Handbook}.

The language of first-order arithmetic $\lone$ has one sort of variables $x, y, z, \dots$, equality, and non-logical symbols $0, 1, +, \cdot, \leq$. 
The language of second-order arithmetic $\ltwo$ is a two-sorted language which is obtained from $\lone$ by adding variables of a second sort $X, Y, Z, \dots$ and a relational symbol $\in$. The language $\ltwo$ does not contain an equality symbol for the second sort. Variables of the first and the second sort are called \textit{first-} and \textit{second-order variables}, respectively.

The classes $\delzn$, $\sigzn$, $\pizn$ are defined in the usual way in terms of first-order quantifier alternations, but allowing second-order free variables. The classes $\sign$, $\pin$, $\deln$ are restrictions of $\delzn$, $\sigzn$, $\pizn$ to the language $\lone$, that is, they consist of formulas without any second-order variables. We use notation like $\sign(A)$ for the set of those $\sigzn$ formulas that contain $A$ as the only second-order parameter. We use these conventions to name theories, e.g, $\iso$ is an $\lone$-theory strictly contained in the $\ltwo$-theory $\iszo$.
If $\Gamma$ is a class of formulas, then $\forall\Gamma$ is the class of universal closures of formulas from $\Gamma$. 

Models for the language $\ltwo$ have the form $(M,\mathcal{X})$, where $M$ is an $\lone$-structure, with its underlying set called the \textit{first-order universe}, and 
$\mathcal{X}\subseteq\mathcal{P}(M)$. We write $\nn$ for the set of natural numbers as formalized within a formal theory (e.g., every model satisfies `$\forall x\,(x\in\nn)$'), whereas $\omega$ denotes the standard natural numbers. A \textit{cut} is a subset $I\subseteq M$ which is an initial segment closed under successor. If $I\neq M$, then $I$ is a \textit{proper cut}. In every model the smallest cut is $\omega$. If $\omega$ is a proper cut, then such a model is called \textit{nonstandard}.

We write `$2^x = y$' and `$2_y(x)=z$' for the usual $\delz$ formulas defining in $\omega$ the graphs of the exponential and superexponential functions, respectively, where the latter is defined recursively by $2_0(x)=x$ and $2_{y+1}(x)=2^{2_y(x)}$. 

A set $X\in\X$ is \textit{finite} if it is bounded, i.e. if $\exists x\,\forall y\,(y\in X \Rightarrow y<x)$ holds. In a formal theory, a set which is not finite is called \textit{unbounded}, whereas the word `infinite' is used only in the metatheory. In a model $(M, \X)\vDash\idzz+\expax$, a subset $X\subseteq M$ is a finite set iff it is \textit{coded} by some number $c\in M$ -- that is, if it has the form $(c)_{\Ack}=\{x\in\nn\colon x\in_{\Ack} c\}$, where $x\in_{\Ack}c$ denotes the usual $\delz$ formula expressing that the $x$-th bit in the binary expansion of $c$ is $1$. We usually identify finite sets with their codes. We write $|c|=y$ for the $\delz$ formula saying that the cardinality of $c$ is $y$.
For a proper cut $I\subseteq M$ we write $\codmi$ for the family of subsets of $I$ that are coded in $M$, that is, $X\in\codmi$ iff $X=I\cap (c)_{\Ack}$ for some $c\in M$.

The system $\rcas$ is usually axiomatized by $\pam$ (a finite list of axioms expressing basic properties of $+$, $\cdot$ and $\leq$), $\idzz$, $\delzo$-comprehension and the axiom $\expax\defeq\forall x\,\exists y\,(x^y=z)$ expressing the totality of the exponential function. It was shown in \cite{Simpson-Smith} that $\rcas$ proves $\bszo$. 
We write $\wkls$ for $\rcas$ extended by weak K\"onig's lemma $\mathsf{WKL}$. By \cite{Simpson-Smith}, if $M\vDash\idz$ and $I\subseteq M$ is closed under exponentiation, then the $\ltwo$-structure $(I, \codmi)$ satisfies $\wkls$.

The theory $\rcas$ is strictly weaker than $\rca$ as it does not imply $\iszo$. 
In fact, by \cite{SimpsonYokoyama}, $\iszo$ is equivalent over $\rcas$ to each of the statements: `every unbounded set $A\subseteq\nn$ is in bijective correspondence with $\nn$' and `every unbounded set $A\subseteq\nn$ has a finite subset with $k$ elements, for arbitrarily large $k$'. 
Thus, $\rcas+\neg\iszo$ implies that there exists an unbounded set $A$ which can be enumerated in increasing order by some proper $\sigzo$-definable proper cut $I$ as $A= \{a_i\}_{i\in I}$ rather than by the whole $\nn$. Conversely, for every $\sigzo$-definable proper cut one can find such a `short’ unbounded set.

\begin{proposition}\label{cofinal set}
Let $(M, \mathcal{X})\vDash\rcas$. For every $\Sigma^0_1$-definable cut $I$ there exists an unbounded set $A\in\mathcal{X}$
which can be enumerated in increasing order as $A=~\!\{a_i\}_{i\in I}$.
\end{proposition}

Working in a model of $\rcas+\neg\iszo$ we will often use an unbounded set like in \cref{cofinal set} to split the first-order universe into $I$-many finite intervals $(a_{i-1}, a_i] = \{x\in\nn\colon a_{i-1}<x\leq a_i\}$. We will use the convention that $a_{-1}=-1$ to get $(a_{-1}, a_0]=\{x\in \nn \colon x\leq a_0\}$. We write $x\in(a_{i-1}, a_i]$ for the formula expressing
`$x$ belongs to the $i$-th interval' and $x=a_i$ for `$x$ is the $i$-th element of $A$'. Note that these are $\delo(A)$-definable binary relations, and the shape of the formulas does not depend on $A$.

The following $\forall\pizth$-definable set will be crucial in our proof of \cref{main theorem_cac}:
\begin{equation}\label{cut 0-1}
\cutzo\vcentcolon = \{ x\in\nn\colon\, \text{every unbounded set has a finite subset of cardinality } x\}.       
\end{equation}
By the correspondence between
unbounded subsets of $\nn$ and $\sigzo$-definable cuts, $\cutzo$ can equivalently be defined as the intersection of all $\sigzo$-definable cuts, so it is itself a cut. 
We observe that $\cutzo$ is closed under multiplication, which is not obvious at first glance, as in all models of $\rcas+\neg\iszo$ there exist $\sigzo$-definable cuts that are not closed even under addition. 

\begin{proposition}\label{proposition:cutzo closed under multiplication}
$\rcas$ proves that the cut $\cutzo$ is closed under multiplication.
\end{proposition}
\begin{proof}
It is enough to check that if every unbounded set has a finite subset of cardinality $x$, then every unbounded set has a finite subset of cardinality $x^2$. So suppose $a^2\notin\cutzo$, and let $S$ be an unbounded set enumerated by a $\sigzo$-cut $I$ that does not have a finite subset of cardinality $a^2$. Note that $a^2$ is above $I$. Let $R$ be the $\delo(S)$-definable set consisting of every $a$-th element of $S$, that is 
\begin{equation*}
    R = \{ x\in S\colon  \exists i<a^2 \,(`x \text{ is the } i\text{-th element of } S\text{'}\,\wedge\, `a\text{ divides }i\text{'}) \}.
\end{equation*}
If $R$ is finite, then the set $S\setminus [0,\, \max(R)]$ is unbounded and 
does not contain a finite subset of cardinality $a$. If $R$ is unbounded, then it itself does not contain a finite subset of cardinality $a$, because otherwise $S$ would have a finite subset of cardinality $a^2$, contrary to our assumption. In any case, $a\notin\cutzo$.
\end{proof}

The cut $\cutzo$ is proper if and only if $\iszo$ fails. In such a case one can ask whether $\cutzo$ is itself $\sigzo$-definable. Both possibilities are consistent with $\rcas+\neg\iszo$. For instance, if $\omega$ is $\sigzo$-definable, then clearly $\cutzo=\omega$. 
For the other possibility, note that the smallest $\sigzo$-definable cut, if it exists, must necessarily satisfy $\iso$, so the theory $\rcas$ + `$\cutzo$ is $\sigzo$-definable' interprets $\iso$. On the other hand, it is known that $\iso$ proves $\Con(\rcas)$, so for G\"odel-style reasons we must have $\rcas\nvdash$ `$\cutzo$ is $\sigzo$-definable'.


To study proof size rigorously we have to fix our proof system. Since in this context it is convenient to think about proofs as finite sequences, our choice is a Hilbert-style calculus. 
Let us note that we lose no generality here, as there are polynomial-time translations between usual Hilbert-style systems and sequent calculi with the cut rule \cite{Eder_book}. 
Officially, we use $\neg$ and $\Rightarrow$ as the only logical connectives and $\forall$ as the only quantifier. We will take advantage of this convention when proving theorems by induction on formula complexity.
However, to enhance readability, we will often use the other connectives and existential quantifier as abbreviations for their usual translations, e.g. $\varphi\wedge\psi$ is shorthand for $\neg(\varphi\Rightarrow\neg\psi)$. When we say `$\theta$ contains $\varphi$ as a conjunct' we mean that $\theta$ is of the above form.
Also, we adopt in the obvious way the definition of formula classes like $\sign$, $\pin$ to our proof system with only one quantifier.

For a given language $\mathcal{L}$, our logical axioms are generated by finitely many schemes as in \cite[Definition 0.10]{HP_Book} with the difference that our only inference rule is modus ponens, and the rule for quantifiers is replaced with the additional axiom schemes $\forall x(\varphi\Rightarrow\psi)\Rightarrow (\forall x\varphi \Rightarrow \forall x\psi)$ and $\theta\Rightarrow\forall x\theta$, where $x$ is not free in $\theta$. 
Since all the languages that we work with have equality, we also include the usual equality axioms. 

We adopt the usual measure of the size of a syntactic object, namely the number of symbols occurring in it, and denote it by vertical lines $\lvert \,\cdot\, \rvert$.

From now on, by a theory we mean just a set of sentences. In particular, we do not require a theory to be deductively closed, since otherwise comparisons of proof size would become trivial.
In the following, it will be convenient for us to work only with finite theories. Thus, we fix finite axiomatizations for $\idz+\expax$, $\rcas$ and $\rca$ using partial satisfaction predicates. Also, we assume that the finitely many instances of $\delz$- and $\delzz$-induction schemes occurring in these axiomatizations are formulated as $\pio$ and $\pizo$ sentences, respectively. 
It can be easily shown that these finite axiomatizations polynomially simulate the usual infinite axiomatizations given by induction and comprehension schemes.


\section{Forcing interpretations}\label{section:FI general}

Recall that an interpretation in the usual sense of a theory $T_1$ in a theory $T_2$ consists in, first, specifying in $T_2$ the domain of the interpretation and translating all formulas of $\mathcal{L}(T_1)$ into $\mathcal{L}(T_2)$, and then showing that the translations of all axioms of $T_1$ are provable in $T_2$. Such an interpretation provides a uniform description of a model of $T_1$ in any model of $T_2$. On the other hand, a forcing interpretation can be intuitively seen as a uniform description of an approximation to a generic model of $T_1$ in any model of $T_2$. 
For a forcing interpretation, one firstly defines a partially ordered set $\mathrm{Cond}$ of \textit{conditions}, a set $\mathrm{Name}$ of \textit{names}, and a relation $\validname{s}{}{v}$ on $\mathrm{Cond}\!\times\!\mathrm{Name}$, pronounced `$s$ forces $v$ to be a valid name'. Then, for every formula $\varphi$ of $\mathcal{L}(T_1)$, one recursively defines the relation $\forces{s}{}{\varphi}$, read as 
`$s$ forces that $\varphi$ holds'. Finally, one shows that $T_2$ proves that every condition forces all axioms of $T_1$.

Our presentation of forcing interpretations follows closely the one from \cite{pfsize}, which in turn is based on Avigad's work \cite{Avigad_paper}, where he introduced the technique and strengthened the $\pioo$-conservativity of $\wkl$ over $\rca$ to polynomial simulation by formalizing Harrington's tree forcing.
The most important difference between our presentation and that from \cite{pfsize} is that we define the forcing translation for all formulas of a given language straightforwardly without a detour through so-called simple formulas, which are roughly translations into a relational language. The detour was made in order to avoid potential ambiguities related to translating formulas with complex terms. However, it should be clear from our constructions that one does not need simple formulas in many natural cases, especially in a situation like ours when terms of the interpreted language are in direct correspondence with terms of the interpreting one.  

We follow standard conventions regarding forcing notation and use lower-case letters $s, s', s''$ etc. as variables for forcing conditions. Our default variables for names are both lower- and (when the interpreted theory is in the two-sorted language) upper-case letters such as $v, w, V, W$. We write $\vv$, $\VV$ for finite tuples of names.

\begin{definition}[{\cite[Definitions 1.5 and 1.6]{pfsize}}]\label{definition:FT}
A \emph{forcing translation $\tau$ of a language~$\mathcal{L}_1$ to a language~$\mathcal{L}_2$}
 consists of $\mathcal{L}_2$-formulas:
  \begin{equation}\label{eq:forcing translation formulas}
   \cond{s}{\tau},\quad
   s'\trianglelefteqslant_\tau s,\quad
   \nam{v}{\tau},\quad
   \validname{s}{\tau}{v},\quad
   s\Vdash_\tau \alpha(\vv),
  \end{equation}
for every atomic $\mathcal{L}_1$~formula~$\alpha(\xx)$,
    such that the following conditions are satisfied.
   \begin{enumerate}
   \renewcommand{\theenumi}{FT\arabic{enumi}}
   \renewcommand{\labelenumi}{(\theenumi)}
   \item For every formula in (\ref{eq:forcing translation formulas}), the only free variables are exactly those shown. 
   \item $s'\trianglelefteqslant_\tau s$ contains
    $s'\in\text{Cond}_\tau\wedge s\in\text{Cond}_\tau$ as a conjunct.
   \item $\validname{s}{\tau}{v}$ contains $s\in\text{Cond}_\tau\wedge \nam{v}{\tau}$ as a conjunct.
   \item $s\Vdash_\tau \alpha(\vv)$ contains\label{item:ftr/f>defd}
    $\validname{s}{\tau}{\vv}$ as a conjunct,
     for every atomic formula $\alpha(\vv)$.

   \item If $\alpha(\uu,v)$ is an atomic $\mathcal{L}_1$~formula and $w$~is a variable,
     then the formulas \label{item:ftr/varsubst}
     \begin{equation*}
      \bigl(s\Vdash_\tau\alpha(\uu,v)\bigr)[w/v]
      \quad\text{and}\quad s\Vdash_\tau\bigl(\alpha(\uu,v)[w/v]\bigr)
     \end{equation*} are the same.
\end{enumerate}
\noindent For complex formulas the forcing relation $\Vdash_\tau$ is defined inductively by the following clauses, where all formulas have exactly the free variables shown.

\begin{enumerate}
\renewcommand{\theenumi}{FT\arabic{enumi}}
\renewcommand{\labelenumi}{(\theenumi)}
\addtocounter{enumi}{5}  

    \item $\forces{s}{\tau}{\neg\varphi(\vv)}$ is \label{item:ftr/neg}
    \begin{equation*}
   \validname{s}{\tau}{\vv}\,\,
   \wedge\,\,\, \forall \lcond{s'}{s}{\tau}\,\bigl(s'\nVdash_\tau\varphi(\vv)\bigr).
  \end{equation*}

   \item $\forces{s}{\tau}{\varphi(\vv, \uu)\Rightarrow\psi(\vv, \ww)}$ is \label{item:ftr/then}
    \begin{multline*}
   \validname{s}{\tau}{\vv} \,\,
   \wedge\,\,\,\validname{s}{\tau}{\uu}\,\,
   \wedge\,\,\, \validname{s}{\tau}{\ww}\,\,
   \wedge\\ \forall \lcond{s'}{s}{\tau}\,\,\exists \lcond{s''}{s'}{\tau}  \bigl(s'\Vdash_{\tau} \varphi(\vv, \uu) \Rightarrow s''\Vdash_{\tau}\psi(\vv, \ww)  \bigr).
  \end{multline*}

  \item $\forces{s}{\tau}{\forall w\,\varphi(w, \vv)}$ is \label{item:ftr/forall} 
    \begin{equation*}
        \validname{s}{\tau}{\vv}\,\wedge\,\,\forall w\,\forall s'\condless{\tau}s\,\exists s''\condless{\tau} s' \bigl(\validname{s'}{\tau}{w}\, \Rightarrow\, s''\Vdash_\tau\varphi(w, \vv) \bigr).
    \end{equation*}

   \end{enumerate}
\end{definition}

\noindent
In the above definition 
$\validname{s}{\tau}{\vv}$ is shorthand for the formula $\bigwedge^n_{i=1} \validname{s}{\tau}{v_i}$,~where $n$ is the length of the tuple $\vv$. We write $\nforces{s}{\tau}{\varphi}$ for
$\neg(\forces{s}{\tau}{\varphi})$.
Expressions like $\forall\cond{s}{\tau}$ and $\forall\nam{v}{\tau}$ are abbreviations for $\forall s\,(\cond{s}{\tau}\Rightarrow \dots)$ and $\forall v\,(\nam{v}{\tau}\Rightarrow \dots)$.
Later, to simplify notation, we will often relax the requirement on displayed free variables and will not distinguish between the free variables occurring in the antecedent and the consequent of an implication.

By straightforward induction on formula complexity one verifies that the technical properties (FT4) and (FT5) essentially hold for all formulas of a given language $\mathcal{L}_1$, see {\cite[Lemma 1.7]{pfsize}}.

The next definition states what is needed for a forcing translation to be a forcing interpretation. To put it simply, the interpreting theory has to prove that the set of forcing conditions is a preorder, that the rules of first-order logic are preserved and that the axioms of an interpreted theory are always forced.

\begin{definition}[{\cite[Definition 1.8]{pfsize}}]\label{definition:FI}
Let $T_1$ and $T_2$ be theories and let $\tau$ be a forcing translation of $\mathcal{L}(T_1)$ to $\mathcal{L}(T_2)$. Then $\tau$ is a \emph{forcing interpretation of $T_1$ in $T_2$} if the following properties are provable in $T_2$.   
\medskip

\noindent The relation $\trianglelefteqslant_\tau$ is a nonempty preorder:
\begin{enumerate}
   \renewcommand{\theenumi}{FI\arabic{enumi}}
   \renewcommand{\labelenumi}{(\theenumi)}
   
    \item $\exists s \bigl( \cond{s}{\tau} \bigr)$,\label{item:fi/non-empty-preorder}
    \item $\forall \cond{s}{\tau} \bigl( \lcond{s}{s}{\tau} \bigr) $,\label{item:fi/preorder:reflexive}
    \item $\forall \cond{s, s', s''}{\tau} \bigl( \lcond{s''}{s'}{\tau} \,\wedge\, \lcond{s'}{s}{\tau} \,\Rightarrow\, \lcond{s''}{s}{\tau} \bigr)$.\label{item:fi/preorder:transitive}
\end{enumerate}
\medskip
\noindent Any generic model is nonempty:
\begin{enumerate}
   \renewcommand{\theenumi}{FI\arabic{enumi}}
   \renewcommand{\labelenumi}{(\theenumi)}   
   \addtocounter{enumi}{3}
    \item $\forall \cond{s}{\tau}\,\, \exists s'\condless{\tau} s\, \exists \nam{v}{\tau} \,\,\validname{s'}{\tau}{v} $.\label{item:fi/non-empty-generic}
\end{enumerate}
\medskip
\noindent The forcing relation $\Vdash_{\tau}$ is monotone:
\begin{enumerate}
   \renewcommand{\theenumi}{FI\arabic{enumi}}
   \renewcommand{\labelenumi}{(\theenumi)}   
   \addtocounter{enumi}{4}
    \item $\forall \cond{s, s'}{\tau}\,\,\forall \nam{v}{\tau} \bigl( \validname{s}{\tau}{v} \wedge\, \lcond{s'}{s}{\tau}\,\Rightarrow\,\validname{s'}{\tau}{v} \bigr)$, \label{item:fi/monotonicity_names}
    
    \item $\forall \cond{s, s'}{\tau}\,\,\forall \nam{\vv}{\tau} \bigl(\forces{s}{\tau}{\alpha(\vv)}  \,\wedge\, \lcond{s'}{s}{\tau} \,\Rightarrow\, \forces{s'}{\tau}{\alpha(\vv)} \bigr)$,\\
    for each atomic formula $\alpha(\xx)$ of $\mathcal{L}(T_1)$.\label{item:fi/monotonicity_atoms}
\end{enumerate}
\medskip
\noindent The axioms of equality are forced, and the values of functions are defined: 
\begin{enumerate}
   \renewcommand{\theenumi}{FI\arabic{enumi}}
   \renewcommand{\labelenumi}{(\theenumi)}   
   \addtocounter{enumi}{6}    
   \item $\forall \cond{s}{\tau}\,\,\forall\nam{v}{\tau} \,\bigl( \validname{s}{\tau}{v}\,\,\Rightarrow\,\, \forces{s}{\tau}{v=v} \bigr)$,\label{item:fi/equality:reflexive}
   \item $\forall \cond{s}{\tau}\,\,\forall\nam{v, v'}{\tau} \,\bigl( \forces{s}{\tau}{v=v'}\,\,\Rightarrow\,\, \forces{s}{\tau}{v'=v} \bigr)$,\label{item:fi/equality:symmetric}
   \item $\forall \cond{s}{\tau}\,\,\forall\nam{v, v', v''}{\tau} \,\bigl( \forces{s}{\tau}{v=v'\,\wedge\,v'=v''}\,\,\Rightarrow\,\, \forces{s}{\tau}{v=v''} \bigr)$,\label{item:fi/equality:transitive}
   \item  $\forall \cond{s}{\tau}\,\,\forall\nam{\vv}{\tau} \,\Bigl( \validname{s}{\tau}{\vv}\,\,\Rightarrow\,\,\\  {\;\;\;\;\;\;\;\;\;\;\;\;\;\;\;\;\;}\bigl(\forall \lcond{s'}{s}{\tau} \,\exists \lcond{s''}{s'}{\tau}\,\exists\nam{w}{\tau}\,(\forces{s''}{\tau}{f(\vv)=w})\,\,\wedge\,\,$\\
   ${\;\;\;\;\;\;\;\;\;\;\;\;\;\;\;\;\;\;\;\;\;\;\;\;\;\;\;\;\;\;\;\;\;\;\;\;}\forces{s}{\tau}{\forall w, w'\,(w=f(\vv) \,\wedge\,w'=f(\vv)\,\,\Rightarrow\,\,w=w')}   \bigr)
   \Bigr)$,\\
   for each function symbol $f$ of $\mathcal{L}(T_1),$\label{item:fi/equality:functions}

 \item $\forall\cond{s}{\tau}\,\forall \nam{\uu, \vv, w}{\tau} \, \bigl(
  \forces{s}{\tau}{w=t(\vv)}\,\Rightarrow   \\ 
{\;\;\;\;\;\;\;\;\;\;\;\;\;\;\;\;\;\;\;\;\;\;\;\;\;\;\;\;\;\;\;\;\;\;\;\;\;\;\;\;\;\;\;\;\;\;\;\;\;\;\;\;\;\;\;\;\;\;\;}\bigl( \forces{s}{\tau}{\alpha(\uu, w)} \,     \Leftrightarrow\, \forces{s}{\tau}{\alpha(\uu, t(\vv))} \bigr)    \bigr),$
   
   for each term $t(\xx)$ and each atomic formula $\alpha(\yy, z)$ of $\mathcal{L}(T_1)$.\label{item:fi/substitution}
    
\end{enumerate}
\medskip
\noindent Density conditions:
\begin{enumerate}
   \renewcommand{\theenumi}{FI\arabic{enumi}}
   \renewcommand{\labelenumi}{(\theenumi)}   
   \addtocounter{enumi}{11}    
   \item $\forall\cond{s}{\tau}\,\forall\nam{v}{\tau}\, \bigl( 
   \forall\lcond{s'}{s}{\tau}\,\exists\lcond{s''}{s'}{\tau}\,(\validname{s''}{\tau}{v})\,\Rightarrow\,\validname{s}{\tau}{v}
   \bigr)$, \label{item:fi/density_names}
   
   \item $\forall\cond{s}{\tau}\,\forall\nam{\vv}{\tau}\, \Bigl( 
   \forall\lcond{s'}{s}{\tau}\,\exists\lcond{s''}{s'}{\tau}\,(\forces{s''}{\tau}{\alpha(\vv)})\,\Rightarrow\,\forces{s}{\tau}{\alpha(\vv)}
   \Bigr)$,\\
   for each atomic formula $\alpha(\xx)$ of $\mathcal{L}(T_1)$. \label{item:fi/density_atoms} 
    
\end{enumerate}
\medskip
\noindent The axioms of $T_1$ are forced:
\begin{enumerate}
   \renewcommand{\theenumi}{FI\arabic{enumi}}
   \renewcommand{\labelenumi}{(\theenumi)}   
   \addtocounter{enumi}{13}    
    \item $\forall\cond{s}{\tau}\,\forces{s}{\tau}{\sigma}$,\\
    for each axiom $\sigma$ of $T_1$.\label{item:fi/forcing axioms}
\end{enumerate}

\end{definition}
\noindent
A forcing translation $\tau$ that satisfies conditions (FI1)-(FI13) is called a \textit{forcing interpretation of $\mathcal{L}(T_1)$ in $T_2$}. By \cite[Proposition 1.14]{pfsize}, it follows that such a $\tau$ is a forcing interpretation of pure logic formulated in the language $\mathcal{L}(T_1)$.

A forcing interpretation $\tau$ is called \textit{polynomial} if there exists a polynomial-time algorithm which constructs proofs of (FI1)-(FI14) for the appropriate inputs. 
Let us note that one needs to describe such an algorithm only for conditions (\ref{item:fi/monotonicity_atoms}), (\ref{item:fi/equality:functions}), (\ref{item:fi/substitution}), (\ref{item:fi/density_atoms}), (\ref{item:fi/forcing axioms}), which are given by schemes. The other conditions are expressed by single sentences, so the construction of their proofs takes a fixed amount of time. 
Note that if $\tau$ is a polynomial forcing interpretation, then the translation $\varphi\mapsto(\forces{s}{\tau}{\varphi})$ is also constructed by a polynomial-time algorithm (it can be simply read off from a proof of (\ref{item:fi/monotonicity_atoms}) for an atomic formula $\alpha$, and for a complex formula $\varphi$ one proceeds by induction, according to clauses (\ref{item:ftr/neg})-(\ref{item:ftr/forall}) from \cref{definition:FT}, and this clearly takes time polynomial in $|\varphi|$).

If the set of forcing conditions has a greatest element, we denote it by $\mathbf{1}$, skipping the implicit subscript $\tau$, as it should always be clear from the context. By \cref{general monotonicity and dencity} below, it follows that $\forces{\mathbf{1}}{\tau}{\varphi}$ is equivalent to $\forall\cond{s}{\tau}(\forces{s}{\tau}{\varphi})$.
Note that if $\varphi$ is a universal statement $\forall x\,\psi(x)$, then $\forall\cond{s}{\tau}\,(\forces{s}{\tau}{\varphi})$ is implied by $\forall\cond{s}{\tau}\,\forall\nam{v}{\tau}\,\bigl(\validname{s}{\tau}{v}\,\Rightarrow\,\forces{s}{\tau}{\psi(v)}\bigr)$. In fact, the opposite implication holds as well, as can be shown using \cref{general monotonicity and dencity}.

Let us note that an interpretation between theories in the usual sense can be seen as a forcing interpretation with a one-element set of conditions. Also, forcing interpretations are closed under definition by cases: if there exist forcing interpretations of $T_1$ in $T_2+\sigma$ and $T_2+\neg\sigma$, then there exists a forcing interpretation of $T_1$ in $T_2$.


The following lemmas show that for any polynomial forcing interpretation, some basic properties can be verified uniformly in polynomial time. We skip the proofs, which can be found in \cite{pfsize} and \cite{kasia:rozprawa}. The first lemma states that for any polynomial forcing interpretation we can feasibly construct a proof that a contradiction is never forced.

\begin{lemma}[{\cite[Lemma 1.10]{pfsize}}]\label{no-contradiction-forced}
Let $\tau$ be a polynomial forcing interpretation of $\mathcal{L}(T_1)$ in $T_2$. Then there exists a polynomial-time algorithm which, given as input a formula $\varphi$ of $\mathcal{L}(T_1)$, outputs a proof in $T_2$ of the sentence:    
\begin{equation}\label{eq:not-forcing-contradiction}
\forall\cond{s}{\tau}\,\forall\nam{\vv}{\tau}\,\neg\bigl( s\Vdash\varphi(\vv) \, \wedge \, s\Vdash\neg\varphi(\vv)
\bigr).    
\end{equation}
\end{lemma}

The next lemma generalizes the monotonicity (\ref{item:fi/monotonicity_atoms}) and density (\ref{item:fi/density_atoms}) conditions to all formulas of the language of the interpreted theory.

\begin{lemma}[{\cite[Lemma 1.11]{pfsize}}]\label{general monotonicity and dencity} 
Let $\tau$ be a polynomial forcing interpretation of $\mathcal{L}(T_1)$ in $T_2$. Then there exists a polynomial-time algorithm, which given as input a formula $\varphi$ of $\mathcal{L}(T_1)$, outputs proofs in $T_2$ of the sentences:

\begin{enumerate}[(a)] 
\item $\forall\cond{s, s'}{\tau}\,\forall\nam{\vv}{\tau}\,\bigl(\bigl(\lcond{s'}{s}{\tau}\,\wedge\,\forces{s}{\tau}{\varphi(\vv)}\bigr)\,\Rightarrow\,\forces{s'}{\tau}{\varphi(\vv)}\bigr)$,
\item $\forall\cond{s}{\tau}\,\forall\nam{\vv}{\tau}\,\bigl(  \forall \lcond{s'}{s}{\tau} \,\exists \lcond{s''}{s'}{\tau} (s''\Vdash\varphi(\vv)) \, \Rightarrow \,  s\Vdash\varphi(\vv)
\bigr)$.    
\end{enumerate}
\end{lemma}

As mentioned in the previous section, our official proof system has only two logical connectives $\neg$ and $\Rightarrow$ and the universal quantifier, so $\varphi\wedge\psi$ and $\exists x\,\varphi(x)$ are shorthand for $\neg(\varphi\Rightarrow\neg\psi)$ and $\neg\forall x\neg\varphi(x)$, respectively. The following simple lemma will be useful later when we force axioms that are expressed most naturally with logical symbols other than our official ones. The proof requires only unfolding the definitions of forcing complex formulas (\ref{item:ftr/neg})-(\ref{item:ftr/forall}) and applying \cref{general monotonicity and dencity}.

\begin{lemma}\label{forcing existential}
Let $\tau$ be a polynomial forcing interpretation of $\mathcal{L}(T_1)$ in $T_2$. Then there exists a polynomial-time algorithm which, given as input formulas $\varphi$ and $\psi$ of $\mathcal{L}(T_1)$, outputs a proof in $T_2$ of the sentences: 
\begin{enumerate}[(a)]
    \item $\forall\cond{s}{\tau}\bigl(\forces{s}{\tau}{\neg(\varphi\Rightarrow\neg\psi)}\,\, \Leftrightarrow\, (\forces{s}{\tau}{\varphi}\,\wedge\,\forces{s}{\tau}{\psi})\bigr)$;\label{eq:forcing conjunction}
    \item $\begin{aligned}[t]
     \forall\cond{s}{\tau}\bigl(\forces{s}{\tau}{\neg\forall x\neg\varphi(x)}\,\, &\Leftrightarrow\\ \forall \lcond{s'}{s}{\tau}\,&\exists \lcond{s''}{s'}{\tau} \,\exists v\,(\validname{s''}{\tau}{v} \,\wedge\, \forces{s''}{\tau}{\varphi(v)})\bigr). \label{eq:forcing existential}  
    \end{aligned}$
    \end{enumerate}
\end{lemma}


The next lemma states the main property of polynomial forcing interpretations. It follows easily from the fact that, given a polynomial forcing interpretation $\tau$, one can construct in polynomial time proofs that logical axioms and modus ponens are forced by any $\cond{s}{\tau}$ \cite[Proposition 1.14]{pfsize}.

\begin{lemma}[{\cite[Corollary 1.15]{pfsize}}]\label{FI-translation of whole proofs}
Let $\tau$ be a polynomial forcing interpretation of $T_1$ in $T_2$. Then there exists a polynomial-time algorithm which, given as input a proof $\delta$ in $T_1$ of a formula $\varphi(\xx)$, outputs a proof in $T_2$ of the sentence:
\begin{equation}\label{eq:forcing theorems}
    \forall\cond{s}{\tau}\,\forall\nam{\vv}{\tau}\,\big( \validname{s}{\tau}{\vv}\,\Rightarrow\, \forces{s}{\tau}{\varphi(\vv)} \bigr).
\end{equation}
\end{lemma}

Finally, we define the reflection property, which expresses the connection between interpreted and interpreting theories. This is a crucial property in establishing a non-speedup result by means of a forcing interpretation.

\begin{definition}[{\cite[Definition 1.16]{pfsize}}]\label{def:reflection}
Let $\tau$ be a polynomial forcing interpretation of $T_1$ in $T_2$ and let $\Gamma$ be a set of sentences in their common language $\mathcal{L}(T_1)\cap\mathcal{L}(T_2)$. Then $\tau$ is \emph{polynomially $\Gamma$-reflecting} if there  exists a polynomial-time algorithm which, given as input a sentence $\gamma\in\Gamma$, outputs a proof in $T_2$ of the sentence:
\begin{equation}\label{eq:reflection}
  \forall\cond{s}{\tau}( \forces{s}{\tau}{\gamma} )\,\Rightarrow\,\gamma .  
\end{equation}
\end{definition}

\begin{theorem}[{Essentially \cite[Section 10]{Avigad_paper}}]\label{thm:Poly-FI-give-poly-simulation}
Let $T_1$ and $T_2$ be theories and let $\Gamma$ be a set of sentences in their common language $\mathcal{L}(T_1)\cap\mathcal{L}(T_2)$. If $\tau$ is a polynomial forcing interpretation of $T_1$ in $T_2$ that is polynomially $\Gamma$-reflecting, then $T_1$ is polynomially simulated by $T_2$ with respect to $\Gamma$.
\end{theorem}


\section{A two-step forcing construction}\label{section:two-step FI}

Our general strategy for strengthening the $\forall\pizth$-conservativity of $\wkls+\cac$ over $\rcas$ to polynomial simulation is similar to the one from \cite{pfsize}, where the non-speedup of $\wkl+\rt$ over $\rca$ was proved by formalizing a model-theoretic argument from \cite{ky:ordinal-valued-ramsey}. That is, we would like to formalize a generic cut construction that can be used to prove the conservation theorem for $\wkls+\cac$.

The general scheme of such a proof would be as follows. 
Given a $\exists\sigzth$ sentence $\varphi\defeq \exists X\,\exists x\,\forall y\,\exists z\,\theta(X, x, y, z)$ consistent with $\rcas$, one takes a countable nonstandard model $(M, \X)\vDash\rcas$ together with $B\in\X$ and $b\in M$ such that $\forall y\,\exists z\,\theta(B, b, y, z)$ holds. Then one defines the following total function:
\begin{equation}\label{eq:function theta_in a model}
    f_\theta(y)=\min\{ z>2^y\colon\, \fale{y'}{y}\, \ele{z'}{z}\, \theta(B, b, y', z') \} 
\end{equation}
and in $\omega$ many steps builds a proper cut $I$ such that the set $\{x\in I\colon \exists y\,(x=f_\theta^{(y)}(b))\}$ is cofinal in $I$, so that the structure $(I, \codmi)$ will satisfy $\varphi$. Moreover, at the $i$-th step of the construction one uses Dilworth's theorem to guarantee that there is a chain or an antichain in $\codmi$ for the $i$-th coded partial order on $I$ (in some fixed enumeration in the metatheory). Note that the structure $(I, \codmi)$ automatically satisfies $\wkls$ (cf. \cref{chapter:preliminaries}).

There is, however, a major problem which does not appear in \cite{ky:ordinal-valued-ramsey}, where the base theory is $\rca$. Namely, in some models of $\rcas$ it may not be possible to perform this construction, as in the absence of $\iszo$ the set $\{f_\theta^{(n)}(b)\colon\, n\in\omega\}$ or even $\{2_n(b)\colon\, n\in\omega\}$ may be unbounded in $M$ -- in the letter case no proper cut in $M$ containing the witness $b$ for $\varphi$ satisfies $\expax$. 
One can easily get around this obstacle simply using compactness, but that approach cannot in general be represented by a forcing interpretation. Thus, we prefer to extend a given model $(M, \X)$ non-cofinally by means of a $\delzo$ ultrapower (a construction first systematically studied by Hirschfeld in \cite{hirschfeld_semiring, hirschfeld_recursive_ultrapowers}), which can naturally be thought of as a forcing construction. Here one takes a nonprincipal ultrafilter in the Boolean algebra of unbounded sets from $\X$ and defines in the usual way the equivalence relation on all total functions on $M$ that are in $\X$.  
Now, for each element $x$ of the ultrapower, the set $\{f_\theta^{(n)}(x)\colon\, n\in\omega\}$ is bounded by $f_\theta^{(\iota)}(x)$, where $\iota$ is the equivalence class of the identity function, so one can build a generic cut satisfying $\cac$ as described in the previous paragraph.

On the syntactical level, we will implement the above idea as two forcing interpretations $\tau_1$ and $\tau_2$, where $\tau_1$ will interpret an auxiliary theory $\sctheory$ capturing those properties of the $\delzo$ ultrapower that are necessary for the generic cut construction. 
The axiom $\sca$ is the following sentence: 
\begin{description}
    \item[$\sca$]  $\mathbb{I}$ is a proper cut closed under multiplication such that, for every number $x$, there is $c>\mathbb{I}$ for which the value $2_c(x)$ exists, 
\end{description}
where $\mathbb{I}$ is a fresh unary predicate symbol (not to be confused with the italic `$I$'). The abbreviation $\sca$ stands for `short cut' because,
intuitively, the cut $\mathbb{I}$ is defined to play a similar role to that of $\omega$ in the above model theoretic argument.

For the first forcing interpretation $\tau_1$ of $\idz+\expax+\sca$, it is enough to define it in $\rcas+\neg\iszo$ rather then in $\rcas$. This is because, by \cref{rt-polynomial simulation}, $\rcas+\iszo$ polynomially simulates $\wkls+\cac$, and polynomial simulations are closed under case distinction.

A natural choice for an interpretation of $\mathbb{I}$ is the cut $\cutzo$, as defined in \cref{chapter:preliminaries}. The reason is that, for every number $x$, the cut $J_x\defeq\{i\in\nn\colon \exists y\, (y=2_i(x))\}$ is $\sigzo$-definable, so it contains $\cutzo$. Now there are two cases to consider: the cut $\cutzo$ may or may not itself be $\sigzo$-definable. If it is not, then we have a straightforward (with no forcing) interpretation of $\idz+\expax+\sca$ in $\rcas+\neg\iszo$: all atomic formulas of $\lone$ translate to themselves and `$t(\xx)\in\mathbb{I}$' is translated to `$t(\xx)\in\cutzo$'. Indeed, this interpretation satisfies the axiom $\sca$ because, by \cref{proposition:cutzo closed under multiplication}, the cut $\cutzo$ is closed under multiplication and, for all numbers $x$, we have $\cutzo\neq J_x$.

In the other case the situation is not that simple: if the cut $\cutzo$ is itself $\sigzo$-definable, then it might happen that it is $J_x$ for some number $x$. Then the axiom $\sca$ would obviously fail under the above interpretation. 
But now we can define a forcing interpretation which takes advantage of this issue. It will simulate the construction of the $\delzo$ ultrapower $M'$ of $M$ in which $\mathbb{I}$ is interpreted as the following cut: 
\begin{equation}\label{eq:nowy sigzo-cut}
\sup\nolimits_{M'}\bigl((\cutzo)^M\bigr)=\{x\in M'\colon \exists y\in M\,(x<y\,\wedge\,M\vDash y\in\cutzo
)\}.    
\end{equation}
One can show that there is a single element $d$ in $M'$, which can be thought of as a diagonal of $(\cutzo)^M$, such that for all numbers $x\in M'$, the value $2_d(x)$ exists. 
\medskip

The outline of the proof of \cref{main theorem_cac} is as follows. We make two case distinctions. First, we consider $\rcas+\iszo$ and $\rcas+\neg\iszo$ separately. In the first case we use \cref{rt-polynomial simulation} from \cite{pfsize}. In the second case we make another case distinction and work with the theories $\rcas+\neg\iszo+\lpc$ and $\rcas+\neg\iszo+\neg\lpc$, where $\lpc$ is the following sentence (`$\mathsf{LPC}$' stands for `least proper cut'):

\begin{description}
    \item[$\lpc$]  {\,\,\,\,\,\,\,\,\,\,\,\,\,\,\,\,\,} The cut $\cutzo$ is $\sigzo$-definable. 
\end{description}

The case of $\rcas+\neg\iszo+\neg\lpc$ is simpler, as this theory admits an almost trivial (non-forcing) interpretation of the auxiliary theory $\sctheory$, which is the identity on $\lone$ and interprets $\mathbb{I}$ as $\cutzo$. Thus we can define a polynomial forcing interpretation of $\wkls+\cac$ directly in $\rcas+\neg\iszo+\neg\lpc$.

In the case of $\rcas+\neg\iszo+\lpc$ we prove polynomial simulation of $\wkls+\cac$ by composing two forcing interpretations. Firstly, we construct a forcing interpretation $\tau_1$ of $\idz+\expax+\sca$ in $\rcas+\neg\iszo+\lpc$, which simulates the model-theoretic construction of a $\delzo$ ultrapower. Then, we define a forcing interpretation $\tau_2$ of $\wkls+\cac$ in $\idz+\expax+\sca$, which follows the model-theoretic proof of $\forall\pizth$-conservativity of $\wkls+\cac$ over $\rcas$. 

Let us note that we cannot, at least if we work with the specific interpretations $\tau_1$ and $\tau_2$, avoid the second case distinction -- it will be clear from our construction that the assumption that $\cutzo$ is $\sigzo$-definable is crucial for $\tau_1$ to work. 


\subsection{Restricted ultrapower}

The forcing conditions of our first forcing interpretation are simply unbounded sets, i.e. second-order elements, ordered by inclusion. 
The set of names consists of those sets which are (graphs of) total functions. Every name is forced to be valid by any condition. We keep the usual forcing notation and use lower-case letters to denote metavariables for conditions and names, but let us stress that they are always second-order objects. 
For a tuple of names $\vv$, we abbreviate the tuple of their values on some number $x$ by $\vv(x)$.
An atomic formula $\alpha(\vv)$ of $\lone$ is forced by a condition $s$ if $\alpha(\vv(x))$ holds for all but finitely many $x\in s$. A condition $s$ forces the value of a term $t(\vv)$ to be in the cut $\mathbb{I}$ if there is an element $i\in\cutzo$ such that the value $t(\vv(x))$ is smaller than $i$ for all but finitely many $x\in s$. 
To simplify notation, we use abbreviations $\forall^\infty x\,\varphi$ and $\exists^\infty x\,\varphi$ for $\exists y \,\forall x\,(y\!<\!x\Rightarrow \varphi)$ and $\forall y\,\exists x\,(y\!<\!x\,\wedge\,\varphi)$, respectively.

\begin{definition}\label{definition_tau 1}
The following list of clauses defines a forcing translation $\tau_1$ from the language $\lone\cup\{\mathbb{I}\}$ to the language $\ltwo$.
\begin{enumerate}[(i)]
    \item $\cond{s}{\tau_1}$ is 
    \begin{equation*}
     \exists^\infty x \, (x\in s);   
    \end{equation*}
   
    \item $\lcond{s'}{s}{\tau_1}$ is
    \begin{equation*}
     \cond{s}{\tau_1}\,\wedge\,\cond{s'}{\tau_1}\,\wedge\,\forall x\,(x\in s'\Rightarrow x\in s); 
    \end{equation*}
    
    \item $\nam{v}{\tau_1}$ is
    \begin{equation*}
        \text{`}v \text{ is a total function'};
    \end{equation*}
    
    \item $\validname{s}{\tau_1}{v}$ is 
    \begin{equation*}
    \cond{s}{\tau_1}\,\wedge\,\nam{v}{\tau_1};    
    \end{equation*}

    \item if $\alpha(\xx)$ is an atomic formula of the form $t_1(\overline{x})=t_2(\overline{x})$ or $t_1(\overline{x})\leq t_2(\overline{x})$, then $\forces{s}{\tau_1}{\alpha({\overline{v}})}$ is
    \begin{equation*}
    \validname{s}{\tau_1}{\vv} \,\wedge\,\, \forall^{\infty}x\,\big(x\!\in \!s\Rightarrow \alpha(\overline{v}(x)\bigr);  
    \end{equation*}    
    
    \item if $\alpha(\xx)$ is an atomic formula of the form $t(\overline{x})\in\mathbb{I}$, then $\forces{s}{\tau_1}{\alpha({\overline{v}})}$ is
    \begin{equation*}
    \validname{s}{\tau_1}{\vv} \,\wedge\,\, \exists i\in\mathrm{I}^0_1\,\,\forall^{\infty}x\,\big(x\in s\Rightarrow t(\vv(x))\leq i\bigr).  
    \end{equation*}
\end{enumerate}
 
\end{definition}

Let us make a few simple observations about $\tau_1$ that will be used later often without mention. 
Firstly, there is a largest forcing condition, namely $\mathbf{1}=\nn$.
Secondly, if $\lcond{s'}{s}{\tau_1}$, then $s'$ is an unbounded subset of $s$. 
Thirdly, it follows immediately from the definition of $\tau_1$ that for every condition $s$, every term $t(\xx)$ and every tuple of names $\vv$ there exists a name $w$ such that $\forces{s}{\tau_1}{w = t(\vv)}$, namely $w(x)=t(\vv(x))$. 
Lastly, the definition  (\ref{item:ftr/forall}) of forcing a universal formula simplifies because every condition forces every name to be valid.


\begin{lemma}\label{tau-1 is a FI}
The forcing translation $\tau_1$ is a polynomial forcing interpretation of the language $\lone\cup\{\mathbb{I}\}$ in $\rcas+\neg\iszo+\lpc$. 
\end{lemma}
\begin{proof}
Let us first note that the conditions from \cref{definition:FI} given by single sentences, i.e. (\ref{item:fi/non-empty-preorder})-(\ref{item:fi/monotonicity_names}), (\ref{item:fi/equality:reflexive})-(\ref{item:fi/equality:transitive}) and (\ref{item:fi/density_names}), follow immediately from the definition of $\tau_1$.
Since there are only three function symbols in the interpreted language, the condition (\ref{item:fi/equality:functions}) is also given by a single sentence and follows easily from the definition of $\tau_1$ as well. 

The schematic conditions (\ref{item:fi/monotonicity_atoms}) and (\ref{item:fi/substitution}) are also unproblematic and their proofs can be constructed by a polynomial-time algorithm.  
Namely, given an atomic formula $\alpha$, the algorithm first finds its forcing translation according to clauses (v) and (vi) of \cref{definition_tau 1}. 
The condition  (\ref{item:fi/monotonicity_atoms}) follows immediately from the definition of $\tau_1$, so the algorithm constructs its proof using a single fixed template in which it substitutes expressions like $\forces{s}{\tau}{\alpha}$ into finitely many blanks. This clearly takes time polynomial in $|\alpha|$.

For (\ref{item:fi/substitution}), the proof goes by induction on the complexity of terms occurring in $\alpha$ and uses the equality axioms in the interpreting theory. For each subterm $r$ of a term occurring in $\alpha$, the proof of the formula $\forall^\infty x\!\in \!s\,\bigl( w(x)=t(\vv(x))\bigr)\Rightarrow \forall^\infty x\!\in \!s\,\bigl(r(\uu(x), w(x))=r(\uu(x), t(\vv(x))) \bigr)$ is constructed just once, so the whole construction of the proof of (\ref{item:fi/substitution}) is polynomial in $|\alpha|$.

The only nonobvious condition is the density property for atomic formulas (\ref{item:fi/density_atoms}). 
There are three cases to consider, depending on whether an atom $\alpha$ is of the form $t_1(\xx)=t_2(\xx)$, $t_1(\xx)\leq t_2(\xx)$ or  $t(\xx)\in\mathbb{I}$. In each case the proof is constructed by contraposition and it should be clear that it does not depend substantially on the terms occurring in an atomic formula $\alpha$ and that it can be constructed in time polynomial in $|\alpha|$. 
We skip the first two cases, as they are much simpler and require only using the definition of $\tau_1$.

So, let $s$ be a condition and $\vv$ be some names. and suppose that $\nforces{s}{\tau_1}{t(\vv)\in\mathbb{I}}$. This means that for every $i\in\mathrm{I}^0_1$ there are unboundedly many $x\in s$ such that $t(\vv(x))>i$.
Recall that we work under the assumption that the cut $\mathrm{I}^0_1$ is $\Sigma^0_1$-definable, so by \cref{cofinal set} there exists an unbounded set $A=\{a_i\}_{i\in \mathrm{I}^0_1}$ indexed by $\mathrm{I}^0_1$. We define recursively a set $s'=\{x_0, x_1, \dots\}$ as follows:
\begin{align*}
x_0 &= \min(s),\\
x_{i+1} &= \min \{y\in s\colon \, y\!>\!x_i \,\wedge\, y\!>\!a_{i} \,\wedge\, t(\overline{v}(y))
>i\}.
\end{align*}
Note that the set $s'$ is unbounded and has $\cutzo$-many elements. Indeed, by the assumption that $\nforces{s}{\tau_1}{t(\vv)\in\mathbb{I}}$, for every $i\in\cutzo$ there exist arbitrarily large numbers $x\in s$ with $t(\vv(x))\!>\!i$. Also, for every $i\in\cutzo$, the $i$-th step of the recursive construction of $s'$ can be performed, since otherwise the set $\{i\in\nn\colon$ $ \exists x\,(x=x_i)\}$ would be a $\sigzo$-cut properly contained in $\cutzo$, contradicting $\lpc$.

Therefore, the set $s'$ is a condition below $s$ and, by its definition, for every $i\in\cutzo$ it holds that $\forall x\in s'(x\!>\!x_i\Rightarrow t(\vv(x))\!>\!i)$. Thus $s'$ does not force ${t(\vv)\in\mathbb{I}}$ and neither does any condition below it. Hence, we obtain $\exists\lcond{s'}{s}{\tau_1}\,\forall\lcond{s''}{s'}{\tau_1}$ $(\nforces{s''}{\tau_1}{\alpha(\vv)})$, as required.
\end{proof}

The following lemma extends condition (v) of \cref{definition_tau 1} to all $\delz$ formulas. It can be seen as a forcing analogue of a restricted version of the Łoś ultraproduct theorem (cf. {\cite[Theorem 2.3]{hirschfeld_recursive_ultrapowers}}). Intuitively, it says that $\delz$ formulas are absolute between the ground model and the generic ultrapower.

\begin{lemma}\label{Delta_zero_Los}
There exists a polynomial-time algorithm which, given as input a $\Delta_0$ formula $\varphi(\zz)$ of $\lone$, outputs a proof in $\rcas+\neg\iszo+\lpc$ of the sentence: 
\begin{equation}\label{eq:los}
\forall\cond{s}{\tau_1}\forall\nam{\vv}{\tau_1}\bigl(\forces{s}{\tau_1}{\varphi(\overline{v})} \,\Leftrightarrow\, \forall^\infty x\!\in\!s \,\,\varphi(\overline{v}(x))\bigr).    
\end{equation}
\end{lemma}

\begin{proof}
We describe a polynomial-time algorithm which constructs a proof of
of (\ref{eq:los}) for a given $\Delta_0$ formula $\varphi$ by recursion on its subformulas. 
The algorithm starts with atomic subformulas of $\varphi$ and uses just one proof template for all atoms. 
Then, given a complex subformula $\psi$, the algorithm takes already produced proofs of (\ref{eq:los}) for the immediate subformulas of $\psi$, and merges these proofs into a proof of (\ref{eq:los}) for $\psi$. This is done by executing one of three fixed subroutines corresponding to the syntactic form of $\psi$, i.e. whether $\psi$ is a negation, an implication or a universally quantified formula. 
For example, if $\psi$ is $\theta\Rightarrow\zeta$, then the algorithm constructs a proof of (\ref{eq:los}) for $\psi$ by inserting $\theta$, $\zeta$ and $\psi$ into a constant number of blanks in the fixed  proof template for implication, and then adjoins the template filled in this way to the previously constructed proofs of (\ref{eq:los}) for $\theta$ and $\zeta$.

We describe informally the subroutine only for a universal subformula, which is the least straightforward case. It should be easily seen from the description that the subroutine takes time polynomial in the size of a given subformula. 

So, let $\psi(\zz)$ be a subformula of $\varphi$ of the form $\fale{y}{t(\zz)}\,\theta(y, \zz)$, which is shorthand for $\forall y\bigl(y\!\leq\! t(\zz)\Rightarrow\theta(y, \zz)\bigr)$, where $y$ is not among $\zz$. Suppose that we have already constructed proofs of (\ref{eq:los}) for $\theta(y, \zz)$ and the atomic subformula $y\!\leq\! t(\zz)$. Let $\vv$ be a tuple of names of the same length as $\zz$. We prove both directions of the equivalence (\ref{eq:los}) for $\psi(\zz)$ by contraposition.
For the ($\Rightarrow$) direction, let us assume that $\exists^\infty x\in s$ $\exists y
\,\bigl(y \leq t(\vv(x))\,\wedge\,\neg\theta(y, \vv(x))\big)$. We define the unbounded set $s'\defeq\{x\in s\colon \exists y
\,\bigl(y \leq t(\vv(x))\wedge\neg\theta(y, \vv(x))\big)\}$ and a function $w$ as follows: 
\begin{equation*}
w(x) =
\begin{cases}
\textrm{least } y\leq t(\vv(x))\, \textrm{ such that } \neg\theta(y, \vv(x)) & \textrm{if  } x\in s'\textrm{,}\\
0 & \textrm{otherwise.}
\end{cases}
\end{equation*}

Both $s'$ and $w$ are $\delzo$-definable, so $s'$ is a condition below $s$ and $w$ is a valid name. From the definitions of $s'$ and the function $w$ we obtain that $\forall^\infty x\!\in\!s' \big(w(x)\!\leq\! t(\vv(x))\wedge \neg\theta(w(x), \vv(x))\big)$. By the proof of (\ref{eq:los}) for the atomic subformula $y\leq t(\zz)$ constructed in the base step of the recursion, and the proof of (\ref{eq:los}) for $\theta(y, \zz)$ constructed in a previous step, we get proofs of $\forces{s'}{\tau_1}{w\leq t(\vv)}$ and $\nforces{s'}{\tau_1}{\theta(w, \vv)}$. But clearly no condition below $s'$ forces $\theta(w, \vv)$, so $\forces{s'}{\tau_1}{\neg\theta(w, \vv)}$. Thus,
by the definition of forcing implication (\ref{item:ftr/then}) and a universal formula (\ref{item:ftr/forall}), we learn that $\nforces{s}{\tau_1}{\forall y\,\bigl(y \leq t(\vv)\Rightarrow\theta(y, \vv)\bigr)}$. 

For the $(\Leftarrow)$ direction, assume that $\nforces{s}{\tau_1}{\forall y\,\bigl(y \leq t(\vv)\Rightarrow\theta(y, \vv)\bigr)}$. 
Then, by (\ref{item:ftr/forall}), there exist a name $w$ and a condition $\lcond{s'}{s}{\tau_1}$ such that no condition $s''$ below $s'$ satisfies $\forces{s''}{\tau_1}{\big(w\leq t(\vv)\Rightarrow\theta(w, \vv)\big)}$. In particular, $\nforces{s'}{\tau_1}{\big(w\leq t(\vv)\Rightarrow\theta(w, \vv)\big)}$.
By the definition of forcing implication (\ref{item:ftr/then}), we learn that there exists a condition $\lcond{s''}{s'}{\tau_1}$ such that $\forces{s''}{\tau_1}{w\leq t(\vv)}$ and $\nforces{s''}{\tau_1}{\theta(w, \vv)}$. 
By the proof of (\ref{eq:los}) for the atomic subformula $y\leq t(\zz)$ constructed in the base step of the recursion, and the proof of (\ref{eq:los}) for $\theta(y, \zz)$ constructed in a previous step,
we obtain that $\exists^\infty x\in s''\,\bigl( w(x) \leq t(\vv(x))\wedge\neg\theta(w(x), \vv(x))\bigr)$. Since $s''$ is unbounded in $s$, we can conclude that $\exists^\infty x\in s\,\, \bigl(w(x) \leq t(\vv(x))\wedge\neg\theta(w(x), \vv(x))\bigr)$, and therefore $\exists^\infty x\in s\,\,\neg\forall y
\,\bigl(y \leq t(\vv(x))\Rightarrow\theta(y, \vv(x))\big)$.

The above construction of the proof of (\ref{eq:los}) for a given $\delz$ formula $\varphi$ takes time polynomial in $|\varphi|$: if $\varphi$ has $k$ subformulas, then our algorithm goes through $k$ stages, where each stage has one of four types, and the time needed to perform a stage of a given type is polynomial in the size of a given subformula of $\varphi$.  
\end{proof}

We finish this section by showing that $\tau_1$ determines a polynomial forcing interpretation of $\idz+\expax+\sca$.

\begin{lemma}\label{tau1-is-polynomial-FI-of-SC}
The forcing translation $\tau_1$ is a polynomial forcing interpretation of $\sctheory$ in $\rcas+\neg\iszo+\lpc$.
\end{lemma}
\begin{proof}
It is enough to show that $\rcas+\neg\iszo+\lpc$ proves that $\mathbf{1}$ forces each axiom of $\sctheory$. Then, the polynomiality of $\tau_1$ will follow by \cref{tau-1 is a FI} and the fact that $\sctheory$ is finitely axiomatized.

So, let us work in $\rcas+\neg\iszo+\lpc$. Using \cref{Delta_zero_Los}, it is straightforward to show that $\mathbf{1}$ forces $\pam + \expax$. We argue only for $\expax$. Take any name $v$. Clearly, the function $w$ defined by $w(x)=2^{v(x)}$ is total, so $\validname{\mathbf{1}}{\tau_1}{w}$. 
Since `$z=2^y$' is a $\delz$ formula, we can use \cref{Delta_zero_Los} to get $\forces{\mathbf{1}}{\tau_1}{w=2^v}$. 

\smallskip
As mentioned in \cref{chapter:preliminaries}, we can assume that $\Delta_0$-induction is given just by one sentence $\forall\xx\,\varphi(\xx)$, where $\varphi$ is $\delz$. We show that $\forces{\mathbf{1}}{\tau_1}{\forall\xx\,\varphi(\xx)}$. So, let $\vv$ be any names. By \cref{definition_tau 1} (iv), $\validname{\mathbf{1}}{\tau_1}{\vv}$. Since $\idz$ holds in the interpreting theory we know that for every number $x$ the formula $\varphi(\vv(x))$ also holds. Thus, by \cref{Delta_zero_Los}, we obtain $\forces{\mathbf{1}}{\tau_1}{\forall\xx\,\varphi(\xx)}$.

\smallskip
Now we show that the axiom $\sca$ is forced. 
We check first that the set $\mathbb{I}$ is forced to be an initial segment, that is, $\mathbf{1}$ forces the sentence $\forall x, y\,\bigl((x\in\mathbb{I}$ $\wedge\,y\!<\!x)\Rightarrow y\in\mathbb{I}\bigr)$.
So, take some names $v$ and $w$ and let $s$ be a condition such that $\forces{s}{\tau_1}{(v\in\mathbb{I}\wedge w<v)}$. 
This means that $\forall^\infty x\!\in\! s\,(w(x)\!<\!v(x))$ and that there is $i\in\cutzo$ such that $\forall^\infty x\!\in\! s\,(v(x)\!<\!i$). Then, clearly, it holds that $\forall^\infty x\!\in\! s\,(w(x)<i)$, so by \cref{definition_tau 1} (vi) we obtain $\forces{s}{\tau_1}{w\in\mathbb{I}}$.

Next we check that $\mathbb{I}$ is forced to be closed under multiplication, i.e. $\mathbf{1}$ forces the sentence $\forall y, z\bigl((y\in\mathbb{I}\wedge z\in\mathbb{I})\Rightarrow yz\in\mathbb{I}\bigr)$.  
So, let $v, w$ be any names and suppose that $s$ is a condition such that $\forces{s}{\tau_1}{(v\in \mathbb{I}\wedge w\in \mathbb{I})}$. Then, there are $i, j\!\in\!\cutzo$ such that $\forall^{\infty}x\!\in\! s\,\bigl(v(x) \leq i\,\wedge\,w(x)\leq j\bigr)$. By \cref{proposition:cutzo closed under multiplication}, we know that the cut $\mathrm{I}^0_1$ is closed under multiplication, so we have $ij\in\cutzo$. Then, $\forall^{\infty}x\in s\,\bigl(v(x)w(x) \leq ij\bigr)$ so, by \cref{definition_tau 1} (vi), we obtain $\forces{s}{\tau_1}{vw\in\mathbb{I}}$.

To show that $\mathbb{I}$ is forced to be a proper cut we find a name $d$ such that $\mathbf{1}$ forces that $d$ is strictly greater than any element of $\mathbb{I}$, i.e. $\forces{\mathbf{1}}{\tau_1}{\forall y\bigl(y\in\mathbb{I}\Rightarrow y<d\bigr)}$.
We use the unbounded set $A=\{a_i\}_{i\in\mathrm{I}^0_1}$ as in \cref{cofinal set} to define the following total function $d$:
\begin{equation}\label{eq:diagonal function}
    d(x) = i,
\end{equation}
where $i$ is the unique element of $\cutzo$ such that $x\in(a_{i-1}, a_{i}]$. Clearly, the function $d$ is $\delo(A)$-definable so it is a valid name.
Now, let $v$ be a name and let $s$ be any condition such that $\forces{s}{\tau_1}{v\in\mathbb{I}}$. Then there exists a number $i\in\mathrm{I}^0_1$ such that $\forall^{\infty}x\in s\,(v(x)\leq i)$. The definition (\ref{eq:diagonal function}) of $d$ guarantees that for all $x\!> \!a_{i}$ it holds that $d(x)>i$, so we get $\forall^{\infty}x\in s\,\bigl(v(x)<d(x)\bigr)$ and then, by \cref{definition_tau 1} (v), we have $\forces{s}{\tau_1}{v<d}$. By the definition (\ref{item:ftr/forall}) of forcing a universal formula, we obtain $\forces{\mathbf{1}}{\tau_1}{\forall y\bigl(y\in\mathbb{I}\Rightarrow y<d\bigr)}$.

Finally, we show that $\mathbf{1}$ forces a strengthening of the last property of $\mathbb{I}$ mentioned by the axiom $\sca$: there exists $z>\mathbb{I}$ such that for all $x$ the value $2_z(x)$ exists.
Using the function $d$ defined above, for every name $v$ we can define a function $w$ by $\delo(A)$-comprehension as follows:
\begin{equation*}
w(x)= 2_{d(x)}(v(x)).
\end{equation*}
The function $w$ is total since each value of $d$ is in $\cutzo$, and for each number $x$ and each $i\in\mathrm{I}^0_1$ the value of $2_i(x)$ exists. Otherwise, for some number $x$ the set $J=\{j\in\nn\colon \exists y\, (y=2_j(x)) \}$ would be a $\sigzo$-definable cut strictly contained in $\mathrm{I}^0_1$ (note that `$y=2_j(x)$' can be expressed by a $\delz$ formula), which is impossible since $\mathrm{I}^0_1$ is the intersection of all $\sigzo$-definable cuts. 

Thus, $w$ is a valid name, and given its definition we can apply \cref{Delta_zero_Los}
to learn that $\forces{\mathbf{1}}{\tau_1}{w=2_d(v)}$. Since $v$ is an arbitrary name, we obtain that $\forces{\mathbf{1}}{\tau_1}{\forall x \,\exists y  \,(y=2_d(x))}$. Finally, we can conclude that $\forces{\mathbf{1}}{\tau_1}{\eg{z}{\mathbb{I}}\,\forall x \,\exists y  \,(y=2_z(x))}$, because we have already shown that $\forces{\mathbf{1}}{\tau_1}{d>\mathbb{I}}$.
\end{proof}


\subsection{Generic cut}\label{section:generic cut}

The definition of our second forcing interpretation is based on \cite[Definition 2.11]{pfsize} with the difference is that our forcing conditions are finite sets of cardinality greater than the cut $\mathbb{I}$ that are exponentially sparse, where a set $s=\{x_0, \dots, x_{n}\}$ is called \textit{exponentially sparse} if for each $i<n$ it holds that $2^{x_i}<x_{i+1}$.
Since the language $\ltwo$ has two sorts of variables, we specify two sets of names and distinguish two cases for the relation of validity. Both first- and second-order names are just arbitrary numbers, but in the latter case we think about names as codes for finite sets and denote them by capital letters. 

The intuition behind this forcing is as follows. If a condition $s$ is in the generic filter, then it will be cofinal in the cut that is being built. In particular, the cut will certainly be above $\min(s)$ and below $\max(s)$. If such an $s$ forces a name $v$ to be valid, then the set $s\cap[v, \infty)$ will also be in the generic filter, and thus $v$ will belong to the cut.

\begin{definition}\label{definition tau_2}
The forcing translation $\tau_2$ from the language $\ltwo$ to the language $\lone\cup\{\mathbb{I}\}$ is defined as follows:
\begin{enumerate}[(i)]
    \item $\cond{s}{\tau_2}$ is
    \begin{equation*}
    \forall x, y\in_{\Ack}s\,(x<y\Rightarrow 2^x<y)\,\,\wedge\,\, |s|>\mathbb{I};
    \end{equation*}
   
    \item $\lcond{s'}{s}{\tau_2}$ is
    \begin{equation*}
     \forall x\,(x\in_{\Ack}s'\Rightarrow x\in_{\Ack}s);  
    \end{equation*}
    
    \item $\nam{v}{\tau_2}$ is
    \begin{equation*}
    v=v;
    \end{equation*}

    \item $\nam{V}{\tau_2}$ is
    \begin{equation*}
    V=V;
    \end{equation*}
    
    \item $\validname{s}{\tau_2}{v}$ is 
    \begin{equation*}
    \cond{s}{\tau_2}\,\wedge\,\nam{v}{\tau_2}\,\wedge\, \ncond{s\cap [0, v]}{\tau_2};    
    \end{equation*}

    \item $\validname{s}{\tau_2}{V}$ is 
    \begin{equation*}
    \cond{s}{\tau_2}\,\wedge\,\nam{V}{\tau_2};    
    \end{equation*}

    \item if $\alpha(\vv)$ is an atomic formula of the form $t_1(\vv)=t_2(\vv)$ or $t_1(\vv)\leq t_2(\vv)$, then $\forces{s}{\tau_2}{\alpha(\vv)}$ is
    \begin{equation*}
    \validname{s}{\tau_2}{\vv} \,\wedge\,\, t_1(\vv)=t_2(\vv)  
    \end{equation*}    or
    \begin{equation*}
    \validname{s}{\tau_2}{\vv} \,\wedge\,\, t_1(\vv)\leq t_2(\vv), 
    \end{equation*} respectively;\label{item:tau2_atom1}
    
    \item if $\alpha(\vv, V)$ is an atomic formula of the form $t(\overline{v})\in V$, then $\forces{s}{\tau_2}{\alpha(\overline{v}, V)}$ is
    \begin{equation*}
    \validname{s}{\tau_2}{\vv}\,\wedge\,\,\validname{s}{\tau_2}{V}\,\wedge\,\, t(\vv)\in_{\Ack}V.  
    \end{equation*}\label{item:tau2_atom2}
\end{enumerate}
\end{definition}

We will often omit the subscript `$_{\Ack}$' when we work in $\sctheory$. This should not lead to any confusion, by our convention to use capital letters to denote names for sets.

Later it will be convenient to have the following list of simple properties of the set of forcing conditions of $\tau_2$. 

\begin{lemma}\label{conditions_2 - basics}
Let $s=\{s_1< \dots< s_c\}$ be a forcing condition of $\tau_2$. Then $\sctheory$ proves the following.
 
\begin{enumerate}[(a)]   
\item $s$ can be split into a disjoint union of two conditions $s=s_1 \sqcup s_2$ such that $\max(s_1)<\min(s_2)$.\label{item:cond_split}

\item Any subset of $s$ with at least $\lceil\sqrt{c}\,\rceil$-many elements is also a condition.\label{item:cond_sqrt}

\item For all names $v$ and $w$, if $\validname{s}{\tau_2}{v}$ and $w<v$, then $\validname{s}{\tau_2}{w}$.\label{item:cond_smaller-names}

\item For every name $v$, if $\validname{s}{\tau_2}{v}$, then $v<\max(s)$.\label{item:cond_less-max} 
\end{enumerate}  
\end{lemma}

\begin{proof}
We reason in $\sctheory$.
\begin{enumerate}[(a)]
\item Let $s_1\defeq\{s_1<\dots< s_{\lfloor\frac{c}{2}\rfloor}\}$ and $s_2\defeq s\setminus s_1$. Clearly, both $s_1$ and $s_2$ are exponentially sparse. Since $|s|=c\leq\lfloor\frac{c}{2}\rfloor+\lfloor\frac{c}{2}\rfloor+1$ and $\mathbb{I}$ is an initial segment closed under addition, we must have $|s_1|=\lfloor\frac{c}{2}\rfloor\notin\mathbb{I}$, so both $s_1$ and $s_2$ are conditions.

\item Let $s'$ be a subset of $s$ with $\lceil\sqrt{c}\,\rceil$-many elements.
Since $|s|<\left(\lceil\sqrt{c}\,\rceil\right)^2$ and $\mathbb{I}$ is an initial segment closed under multiplication, we have $\lceil\sqrt{c}\,\rceil\notin\mathbb{I}$. Clearly, $s'$ is also exponentially sparse, so it is a condition.

\item If $w<v$, then $s\cap[0, w]\subseteq s\cap[0, v]$, and thus  $|s\cap[0, w]|\leq|s\cap[0, v]|\in\mathbb{I}$, so $w$ is not a condition.

\item If $v\geq\max(s)$, then $s\cap[0, v]=s\in\mathrm{Cond}_{\tau_2}$.\qedhere 
\end{enumerate}
\end{proof}
Note that all the clauses of the above lemma are formulated as single sentences and not as schemes, so including (the formal version of) their proofs in any other proof we construct will only increase the time complexity of the relevant algorithm by a fixed additive constant.

\begin{lemma}\label{tau-2 is a FI}
$\tau_2$ is a polynomial forcing interpretation of the language $\ltwo$ in the theory $\sctheory$. 
\end{lemma}

\begin{proof}
We first check that $\sctheory$ proves that $\tau_2$ satisfies the conditions from \cref{definition:FI} given by single sentences, that is, (\ref{item:fi/non-empty-preorder})-(\ref{item:fi/monotonicity_names}), (\ref{item:fi/equality:reflexive})-(\ref{item:fi/equality:transitive}) and (\ref{item:fi/density_names}).

The relation $\trianglelefteqslant_{\tau_2}$ is clearly a preorder because it coincides with set inclusion. To see that its field is nonempty, note that the axiom $\sca$ guarantees that there are in fact unboundedly many forcing conditions: for any number $x$, if $y$ is a number above the cut $\mathbb{I}$ such that the value $2_y(x)$ exists, then for $y'\colon\!\!=\!2\cdot\lfloor\frac{y}{2}\rfloor$ the set $s=\{2_2(x), 2_4(x), \dots, 2_{y'}(x)\}$ is exponentially sparse and has $y'$-many elements, where $y'>\mathbb{I}$ because $\mathbb{I}$ is closed under multiplication. The code for $s$ exists because it is bounded by the value $2_{y'+1}(x)$. Thus, the conditions (\ref{item:fi/non-empty-preorder})-(\ref{item:fi/preorder:transitive}) are satisfied.

For (\ref{item:fi/non-empty-generic}), which says that any generic model is nonempty, note that every condition forces each standard natural number to be a valid name. 
To see that the monotonicity condition for names (\ref{item:fi/monotonicity_names}) holds, assume that $\lcond{s'}{s}{\tau_2}$ and $\validname{s}{\tau_2}{v}$, which means that $|s\cap [0, \dots, v]|\in\mathbb{I}$. Then, by the definition of $\trianglelefteqslant_{\tau_2}$, we get $s'\cap [0, \dots, v]\subseteq s\cap [0, \dots, v]$, so $|s'\cap [0, \dots, v]|\in\mathbb{I}$ and thus $\validname{s'}{\tau_2}{v}$, as required.

The conditions (\ref{item:fi/equality:reflexive})-(\ref{item:fi/equality:transitive}) expressing that equality is an equivalence relation follow immediately from the definition of forcing atomic formulas. 
Concerning the density condition \eqref{item:fi/density_names}, we only need some care in the case of first-order names, because the definition of $\validname{s}{\tau_2}{v}$ is nontrivial. 
We reason by contraposition. Let $s$ and $v$ be such that $\nforces{s}{\tau_2}{v}$. This means that the set $s\cap[0, v]$ is a condition, i.e. it has more than $\mathbb{I}$ many elements. But then for any condition $\lcond{s'}{s\cap[0, v]}{\tau_2}\!\trianglelefteqslant\! s$ it holds that $s'\cap[0, v]=s'$, so it cannot force $v$ to be a valid name. Hence, the antecedent of \eqref{item:fi/density_names} fails for $s$ and $v$.

The condition (\ref{item:fi/equality:functions}), saying that the values of functions are well-defined, is also given by a single sentence because there are only three function symbols in the interpreted language. In fact, we can show something stronger than (\ref{item:fi/equality:functions}): there exists a polynomial-time algorithm which, given as input a term $t(\xx)$ of $\ltwo$, outputs a proof of the sentence:
\begin{multline}\label{eq:valid-values-of-terms}
 \!\!\!\!\!\forall \cond{s}{\tau_2}\,\forall\nam{\vv}{\tau_2} \,\Bigl( \validname{s}{\tau_2}{\vv}\,\Rightarrow \bigl(\exists\nam{w}{\tau_2}\,(\validname{s}{\tau_2}{w}\,\wedge\, \forces{s}{\tau_2}{t(\vv)=w}) \\ \wedge\,
   \forces{s}{\tau_2}{\forall w, w'\,(w=t(\vv) \,\wedge\,w'=t(\vv)\,\,\Rightarrow\,\,w=w')}   \bigr)
   \Bigr){\;\;\;\;\;\;\;} 
\end{multline}
We describe how to construct a proof of the above sentence for any term $t(\xx)$ of $\ltwo$. One picks arbitrary $s$ and $\vv$ and assumes that $\validname{s}{\tau_2}{\vv}$. The proof of the uniqueness of the value $t(\vv)$ requires only invoking the definition of forcing for atomic formulas as well as clauses (\ref{item:ftr/then}) and (\ref{item:ftr/forall}). This is a single template in which one has to substitute the term $t$ a fixed number of times. 

The proof of the existence of the value $t(\vv)$ is constructed by recursion on subterms of $t$. 
The base step is trivial as we have to consider subterms which are either variables or numerals. 
For the recursive step, assume that the algorithm has already constructed proofs of the existence of the values $v_1$, $v_2$ for subterms $r_1(\vv)$, $r_2(\vv)$. 
By \cref{conditions_2 - basics} \ref{item:cond_smaller-names}, it is enough to check that there is a valid name $w$ for the complex term $v' \cdot v'+1$, where $v'=\max\{v_1, v_2\}$. 
So, from the assumption $\validname{s}{\tau_2}{v'}$ we know that 
$|s\cap[0, v']|$ is in $\mathbb{I}$. Since $s$ is exponentially sparse, it holds that $|s\cap[0, v'^2+1]|\leq |s\cap[0, v']|+1\in\mathbb{I}$, so $s\cap[0, v'^2+1]$ is not a condition either and thus $\validname{s}{\tau_2}{v'^2+1}$.
The proof of (\ref{eq:valid-values-of-terms}) for $t$ is finished by recalling \cref{forcing existential} (a) and noting that $s$ and $\vv$ were arbitrary. 

The schematic conditions (\ref{item:fi/monotonicity_atoms}), (\ref{item:fi/substitution}) and (\ref{item:fi/density_atoms}), which concern forcing atomic formulas, follow immediately from the definition of $\tau_2$. The algorithm constructs their proofs using three fixed templates, in which it substitutes expressions like $\forces{s}{\tau_2}{\alpha(\vv)}$ and, in case of (\ref{item:fi/monotonicity_atoms}) and (\ref{item:fi/density_atoms}), also previously constructed proofs for (\ref{item:fi/monotonicity_names}) and (\ref{item:fi/density_names}), respectively. This clearly takes time polynomial in the size of a given atomic formula $\alpha$.
\end{proof}


The next lemma extends clause (viii) of \cref{definition tau_2} and says that a $\delzz$ formula is forced if and only if its Ackermann translation holds. Here, by \emph{the Ackermann translation} of a $\delzz$ formula $\varphi$ we mean an $\lone$-formula 
$\varphi_{\Ack}$ which is obtained from $\varphi$ by replacing all atomic subformulas of the form `$t\in X$' with `$t\in_{\Ack}X$', where the number variable $X$ does not occur in $\varphi$ -- recall that here we use capital letters for numbers that occur as names for sets. 
To keep the notation simple, we do not distinguish between an $\ltwo$-formula and its Ackermann translation.

\begin{lemma}\label{t2-delta0-elementarity}
There exists a polynomial-time algorithm which, given as input a $\delzz$ formula $\varphi(\xx, \XX)$ of $\ltwo$, outputs a proof in $\sctheory$ of the sentence: 
\begin{equation}\label{eq:tau2-delta00-elementarity}
\forall\cond{s}{\tau_2}\,\forall\nam{\vv, \VV}{\tau_2}\bigl(\validname{s}{\tau_2}{\vv}\,\Rightarrow\bigl(\forces{s}{\tau_2}{\varphi(\vv, \VV)} \Leftrightarrow \varphi(\vv, \VV)\bigr)\bigr).    
\end{equation}
\end{lemma}

\begin{proof}
We show how to construct a proof of (\ref{eq:tau2-delta00-elementarity}) for a given $\delzz$ formula $\varphi$ by recursion on its subformulas. The argument that the construction can be performed by a polynomial-time algorithm is similar to the one in the proof of \cref{Delta_zero_Los}.

The base step for atomic subformulas of $\varphi$ is guaranteed by clauses \ref{item:tau2_atom1} and \ref{item:tau2_atom2} of \cref{definition tau_2}. 
There are three cases of a complex subformula $\psi(\xx, \XX)$, depending on whether $\psi$ is a negation, an implication or a universally quantified formula. The cases of negation and implication are straightforward and use only conditions (\ref{item:ftr/neg}) and (\ref{item:ftr/then}), as well as \cref{no-contradiction-forced}. We discuss only the case of a universal subformula of the form $\fale{y}{t(\xx)\,\theta(y, \xx, \XX)}$, which is shorthand for $\forall y \bigl(y\leq t(\xx)\Rightarrow\theta(y, \xx, \XX)\bigr)$, where $y$ is not among $\xx$. 

Suppose that we have already constructed a proof of (\ref{eq:tau2-delta00-elementarity}) for $\theta(y, \xx, \XX)$. Let $\vv$ and $\VV$ have the same length as $\xx$ and $\XX$, respectively, and let $\validname{s}{\tau_2}{\vv}$. 
Assume that $\forces{s}{\tau_2}{\forall y \bigl(y\leq t(\vv)\Rightarrow\theta(\vv, \VV)\bigr)}$. Take some $u\leq t(\vv)$.
By the proof of \cref{tau-2 is a FI}, specifically by the proof of \eqref{eq:valid-values-of-terms}, there exists a name $w$ such that $\validname{s}{\tau_2}{w}$ and $\forces{s}{\tau_2}{w=t(\vv)}$.
By the definition of forcing for atomic formulas, we know that $w=t(\vv)$ so, by \cref{conditions_2 - basics} \ref{item:cond_smaller-names}, we get $\validname{s}{\tau_2}{u}$. Again, by the definition of forcing for atomic formulas we get $\forces{s}{\tau_2}{u\leq t(\vv)}$. By (\ref{item:ftr/then}) and (\ref{item:ftr/forall}), there exists $\lcond{s'}{s}{\tau_2}$ such that $\forces{s'}{\tau_2}{\theta(u, \vv, \VV)}$ so, by the previously constructed proof of (\ref{eq:tau2-delta00-elementarity}) for $\theta(y, \xx, \XX)$, we obtain $\theta(u, \vv, \VV)$. Since $u$ was arbitrary, we conclude that $\fale{y}{t(\xx)\,\theta(y, \xx, \XX)}$ holds.

Conversely, assume that $\forall y \bigl(y\leq t(\vv)\Rightarrow\theta(\vv, \VV)\bigr)$ holds. Take a name $u$ and a condition $\lcond{s'}{s}{\tau_2}$ such that $\validname{s'}{\tau_2}{u}$ and $\forces{s'}{\tau_2}{u\leq t(\vv)}$. By the definition of forcing for atomic formulas we have $u\leq t(\vv)$ so, by our assumption, $\theta(u, \vv, \VV)$ holds. By the previously constructed proof of (\ref{eq:tau2-delta00-elementarity}) for $\theta(y, \xx, \XX)$, we get $\forces{s'}{\tau_2}{\theta(u, \vv, \VV)}$ so, by the definitions of forcing implication (\ref{item:ftr/then}) and a universal sentence (\ref{item:ftr/forall}), we obtain that $\forces{s}{\tau_2}{\forall y \bigl(y\leq t(\vv)\Rightarrow\theta(\vv, \VV)\bigr)}$.
\end{proof}

It remains to prove that all the axioms of $\wkls+\cac$ are forced. 

\begin{lemma}\label{tau2-is-polynomial-FI-of-CAC}
The forcing translation $\tau_2$ is a polynomial forcing interpretation of $\wkls+\cac$ in the theory $\sctheory$.
\end{lemma}

\begin{proof}
It is enough to show that $\sctheory$ proves that every condition of \\ $\tau_2$ forces every axiom of $\wkls+\cac$. Then the polynomiality of $\tau_2$ will follow by \cref{tau-2 is a FI} and by our assumption from \cref{chapter:preliminaries} that $\rcas$ is finitely axiomatized.

So, we reason in $\sctheory$.
It follows immediately from \cref{t2-delta0-elementarity} and the definition of forcing a universal formula (\ref{item:ftr/forall}) that the axioms of $\pam$ are forced by every condition.
Let us check that $\expax$ is also forced. Suppose that $\validname{s}{\tau_2}{v}$, where $s=\{s_1< \dots< s_c\}$. Then $s\cap[0, v]=\{s_1, \dots, s_i\}$ is not a condition, so $i\in\mathbb{I}$.
Since $v<s_{i+1}$ and $s$ is exponentially sparse, we have $2^v< s_{i+2}$,
and hence $s\cap[0, 2^v]\subseteq\{s_1, \dots, s_{i+1}\}$. Thus, we get
$|s\cap[0, 2^v]|\leq i+1\in\mathbb{I}$,
and so $\validname{s}{\tau_2}{2^v}$.
By \cref{t2-delta0-elementarity}, this implies $\forces{s}{\tau_2}{\expax}$ because `$y=2^x$' is a $\delz$ formula.

\smallskip
For $\delzz$-induction, 
by our choice of the axiomatization of $\rcas$, we can assume that it is given by a single sentence $\forall \XX\,\forall \xx\,\varphi(\XX, \xx)$, where $\varphi$ is $\delzz$. Let $\VV$ and $\vv$ be tuples of names of the same length as $\XX$ and $\xx$, and let $s$ be a condition such that $\validname{s}{\tau_2}{\vv}$ holds. Since we work under $\idz$, we know that (the Ackermann translation of) $\varphi(\VV, \vv)$ holds. Therefore, by \cref{t2-delta0-elementarity}, we obtain $\forces{s}{\tau_2}{\varphi(\VV, \vv)}$, as required.

\smallskip
To prove that $\mathsf{WKL}$ and $\delzo$-comprehension are forced we follow a fragment of the proof of Lemma 2.15 in \cite{pfsize}.
Recall that the $\Sigma^0_1$-separation principle is the following scheme:
\begin{multline*}
\forall \zz\,\forall\ZZ \bigl(\forall x \bigl(\exists y \,\varphi_1(x, y, \zz, \ZZ) \Rightarrow \neg\exists y\,\varphi_2(x, y, \zz, \ZZ)\bigr) \Rightarrow \\  \exists X\forall x \bigl((\exists y\,\varphi_1(x, y, \zz, \ZZ)\Rightarrow x\in X) \wedge (\exists y\,\varphi_2(x, y, \zz, \ZZ)\Rightarrow x\notin X) \bigr)\bigr),     
\end{multline*}
where $\varphi_1$, $\varphi_2$ are $\delzz$. By \cite[Lemma~IV.4.4]{Simpson_Book}), this principle implies over $\rca$ both $\mathsf{WKL}$ and $\delzo$-comprehension, and it is easy to check that the implication remains valid over $\rcas$ (cf. the proof of Lemma 3.2 in \cite{fkwy}). 
Note that we only need to show that every condition forces a fixed finite number of instances of $\Sigma^0_1$-separation which is needed to prove $\mathsf{WKL}$ and the finitely many instances of $\delzo$-comprehension occurring in our axiomatization of $\rcas$. Then, by \cref{FI-translation of whole proofs}, we can conclude that every condition forces $\mathsf{WKL}$ and $\delzo$-comprehension as well.

So, let $s$, $\vv$ and $\VV$ be such that $\validname{s}{\tau_2}{\vv}$ and $\forces{s}{\tau_2}{\forall x \bigl(\exists y\,\varphi_1(x, y, \vv, \VV) \Rightarrow}$ ${\neg\exists y\,\varphi_2(x, y, \vv, \VV)\bigr)}$. 
Define the following finite set:
\begin{equation*}
W\defeq \{x<\max(s)\colon \el{y}{\max(s)}
\bigl(\varphi_1(x, y, \vv, \VV) \,\wedge\, \fal{z}{y} \neg\varphi_2(x, z, \vv, \VV)\bigr) \}.    
\end{equation*}
We check that already $s$ forces the following universal sentence:
\begin{equation}\label{conclusion of separation principle}
\forall x\bigl( (\exists y\,\varphi_1(x, y, \vv, \VV)\Rightarrow x\in W) \,\wedge\, (\exists y\,\varphi_2(x, y, \vv, \VV)\Rightarrow x\notin W) \bigr).
\end{equation} 
Let $s'$ and $u$ be such that $\lcond{s'}{s}{\tau_2}$ and $\validname{s'}{\tau_2}{u}$. By \cref{forcing existential} (a), it is enough to show that $s'$ forces each of the conjuncts of the instance of (\ref{conclusion of separation principle}) for $u$. For the first one, we assume, purely for simplicity of notation, that already $\forces{s'}{\tau_2}{\exists y\,\varphi_1(u, y, \vv, \VV)}$. We show that $u\in W$ which, by \cref{t2-delta0-elementarity}, implies that also $\forces{s'}{\tau_2}{u\in W}$.

By \cref{forcing existential} (b), there exist a condition $\lcond{s''}{s'}{\tau_2}$ and a name $w$ such that $\validname{s''}{\tau_2}{w}$ and $\forces{s''}{\tau_2}{\varphi_1(u,w, \vv, \VV)}$. Then, by \cref{t2-delta0-elementarity}, $\varphi_1(u,w, \vv, \VV)$ holds. 
By \cref{conditions_2 - basics} \ref{item:cond_less-max}, it holds that $w<\max(s'')$. On the other hand, by \cref{general monotonicity and dencity} (a), $\forces{s''}{\tau_2}{\exists y\,\varphi_1(u, y, \vv, \VV) \Rightarrow}$ ${\neg\exists y\,\varphi_2(u, y, \vv, \VV)}$. Thus, any $z$ such that $\varphi_2(u, z, \vv, \VV)$ holds cannot be forced by $s''$ to be a valid name, because otherwise $s''$ would also force $\varphi_2(u, z, \vv, \VV)$. In particular, any such $z$ has to be greater than $w$, so $u$ satisfies the condition defining $W$, as required.

For the second conjunct of (\ref{conclusion of separation principle}), we assume as previously that already $\forces{s'}{\tau_2}{\exists y\,\varphi_2(u, y, \vv, \VV)}$, and we show that $u\notin W$, which implies $\forces{s'}{\tau_2}{u\notin W}$ by \cref{t2-delta0-elementarity}. As above, by \cref{forcing existential} (b) and \cref{conditions_2 - basics} \ref{item:cond_less-max}, we get a condition $\lcond{s''}{s'}{\tau_2}$ and a name $z<\max(s'')$ such that 
$\validname{s''}{\tau}{z}$ and  $\forces{s''}{\tau_2}{\varphi_2(u,z, \vv, \VV)}$. Thus, by \cref{t2-delta0-elementarity}, we have that $\varphi_2(u,z, \vv, \VV)$ holds. 
 
Towards a contradiction, suppose that $u\in W$.  
Then, there must be some number $w<\max(s)$ such that $\varphi_1(u, w, \vv, \VV)$ holds but for every $k<w$ we have $\neg\varphi_2(u, k, \vv, \VV)$.  
In particular, we have $w\leq z$ so, by \cref{conditions_2 - basics} \ref{item:cond_smaller-names}, we get $\validname{s''}{\tau_2}{w}$.
But now, by \cref{t2-delta0-elementarity}, we learn that 
$\forces{s''}{\tau_2}{\varphi_1(u, w, \vv, \VV)}$ so, obviously, $\forces{s''}{\tau_2}{\exists y\,\varphi_1(u, y, \vv, \VV)}$. This is a contradiction as, by \cref{general monotonicity and dencity} (a), $\forces{s''}{\tau_2}{\exists y\,\varphi_1(u, y, \vv, \VV) \Rightarrow}$ ${\neg\exists y\,\varphi_2(u, y, \vv, \VV)}$.
Therefore, $u\notin W$. This completes the proof that $\mathsf{WKL}$ and $\delzo$-comprehension are forced.

\smallskip

Finally, we check that $\cac$ is forced. We will use the more intuitive symbol $\preccurlyeq$ for a finite set that codes a partial order.
Let $s$ be such that $s\Vdash_{\tau_2}$`$\preccurlyeq$ is a partial order on $\mathbb{N}$'. We will find a condition $\lcond{s^*}{s}{\tau_2}$ which forces the sentence `There exists an unbounded chain or antichain in $\preccurlyeq$'.
By \cref{conditions_2 - basics} \ref{item:cond_split}, we can split $s$ into a disjoint union $s=s_1\sqcup s_2$ such that $\max(s_1)<\min(s_2)$ and both $s_1, s_2$ are conditions. Now, for every $v\in s_1$ we have $s_2\cap[0, v]=\emptyset$, so $\validname{s_2}{\tau_2}{v}$. By \cref{general monotonicity and dencity} (a), $s_2$ also forces `$\preccurlyeq$ is a partial order on $\mathbb{N}$' so, by \cref{t2-delta0-elementarity}, we learn that $\preccurlyeq$ is a partial order on $s_1\times s_1$.
Let $c$ be the largest number such that $|s_1|\geq c(c-1)$. Then, from \cref{thm:Dilworth} (Dilworth's theorem), which is easily provable in $\idz+\expax$ by elementary finite combinatorics, it follows that there exists $s^*\subseteq s_1$ such that $|s^*|=c$ and $s^*$ is a chain or an antichain in $\preccurlyeq\restriction\!s_1\times s_1$. 
Note that, by \cref{conditions_2 - basics} \ref{item:cond_sqrt}, $s^*$ is also a condition. By \cref{t2-delta0-elementarity}, we obtain $\forces{s^*}{\tau_2}{}$`$s^*$ is a chain or antichain in $\preccurlyeq$'.
Let us stress that `$s^*$' on the right of the previous formula occurs as a second-order name, so clearly $\validname{s^*}{\tau_2}{s^*}$. 

The last thing to show is that $s^*$ forces itself to be an unbounded set, that is, $\forces{s^*}{\tau_2}{\forall x\,\exists y\in s^*(x<y)}$.
So, pick some condition $\lcond{s'}{s^*}{\tau_2}$ and a name $v$ such that $\validname{s'}{\tau_2}{v}$. 
Then $s'\cap[0, v]$ is not a condition. By \cref{conditions_2 - basics} \ref{item:cond_less-max}, there are some $s_i, s_{i+1}\in s'=\{s_1, \dots, s_c\}$ such that $s_i\leq v < s_{i+1}$. 
Note that $|s'\cap[0, v]|=i\in\mathbb{I}$, so clearly $|s'\cap[0, s_{i+1}]|=i+1\in\mathbb{I}$ and thus $\validname{s'}{\tau_2}{s_{i+1}}$. Since $s'\subseteq s^*$, we have $s_{i+1}\in s^*$, and so, by \cref{t2-delta0-elementarity}, we get $\forces{s'}{\tau_2}{(s_{i+1}\in s^*\,\wedge\,v<s_{i+1})}$.
This finishes the proof that $\cac$ is forced by every condition.
\end{proof}


\subsection{Completing the proof}

In this last section we finally prove 
\cref{main theorem_cac}. Our main task is to combine the forcing interpretations $\tau_1$ and $\tau_2$ to obtain a polynomial simulation of $\wkls+\cac$ in $\rcas+\neg\iszo+\lpc$.
Then we will complete the proof by adding the remaining cases of $\rcas+\neg\iszo+\neg\lpc$ and $\rcas+\iszo$.

We need two technical lemmas. The first one says that our two forcing interpretations $\tau_1$ and $\tau_2$ can be composed. The second one states that the composition is $\forall\pizth$-reflecting.  

\begin{lemma}\label{composition -lemma 1}
There exists a polynomial-time algorithm which, given as input a proof $\delta$ of a~sentence $\sigma$ in $\wkls+\cac$, outputs a proof in $\rcas+\neg\iszo+\lpc$ of the sentence: 
\begin{equation*}
\forces{\mathbf{1}}{\tau_1}{(\forall \cond{s}{\tau_2}\,(\forces{s}{\tau_2}{\sigma}))}.
\end{equation*}
\end{lemma}

\begin{proof}
Let $\delta$ be a proof of a sentence $\sigma$ in $\wkls+\cac$. By \cref{tau2-is-polynomial-FI-of-CAC} and \cref{FI-translation of whole proofs}, there exists a polynomial-time algorithm which, given as input $\delta$, returns a proof $\delta'$ in $\sctheory$ of the sentence $\forall\cond{s}{\tau_2}(\forces{s}{\tau_2}{\sigma})$. Now, by \cref{tau1-is-polynomial-FI-of-SC} and again \cref{FI-translation of whole proofs}, one can apply another polynomial-time algorithm which on input $\delta'$ outputs a proof in $\rcas+\neg\iszo+\lpc$ of the sentence $\forces{\mathbf{1}}{\tau_1}{(\forall \cond{s}{\tau_2}\,(\forces{s}{\tau_2}{\sigma}))}$, as required.
\end{proof}

\begin{lemma}\label{composition - lemma 2}
There exists a polynomial-time algorithm which, given as input an $\ltwo$-sentence $\exists X\exists x\,\varphi(X, x)$, where $\varphi(X, x)$ is $\pizt$, outputs a proof in $\rcas+\neg\iszo+\lpc$ of the sentence: 
\begin{equation}\label{eq:reflection lemma}
\exists X\exists x\,\varphi(X, x)\,\Rightarrow\,\forces{\mathbf{1}}{\tau_1}{\bigl(\exists\cond{s}{\tau_2}\,\exists\nam{V, v}{\tau_2}\,(\forces{s}{\tau_2}{\varphi(V, v)})\bigr)}.    
\end{equation} 
\end{lemma}
\begin{proof}
We describe informally how to construct a proof of (\ref{eq:reflection lemma}) for an arbitrary sentence $\exists X\exists x\,\varphi(X, x)$ with $\varphi(X, x)$ being $\pizt$. It should be easily seen that the main part of the proof is constructed from a single template for all such sentences, into which one substitutes a fixed number of times the sentence `$\exists X\exists x\,\varphi(X, x)$' or other expressions obtained from it in polynomial-time. This is complemented by a few auxiliary procedures that can be seen to be polynomial-time, which we will comment on in appropriate places below.

So, let $\exists X\,\exists x\,\varphi(X, x)$ be an $\ltwo$-sentence, where $\varphi(X, x)$ is $\pizt$. 
For technical reasons that will become clear below, we replace the subformula $\varphi(X, x)\defeq\forall y\,\exists z\,\theta(X, x, y, z)$ with a $\pizt$ formula $\varphi^*(X, x)\defeq\forall y\,\exists w\,\ele{z}{w}\,\theta^*(X, x, y, z, w)$, where all quantifiers in $\theta^*$ are bounded by $w$. The algorithm building the proof of \eqref{eq:reflection lemma} starts its work with the construction of a proof in $\rcas$ of the equivalence:
\begin{equation}\label{eq:delta0-special-form}
\forall X\,\forall x\,\bigl(\varphi(X, x)\Leftrightarrow\varphi^*(X, x)\bigr)
\end{equation}
by recursion on subformulas of $\theta(X, x, y, z)$. We skip the details of this procedure, but it should be easily seen that it can be carried out in time polynomial in $|\theta(X, x,y, z)|$. To simplify notation, from now on we will assume that already $\varphi(X, x)$ is in the above form where each quantifier in the $\delzz$ matrix is bounded by a variable.

Next, the algorithm writes the following definition of a (possibly partial) function of the variable $y$ with parameters $X$ and $x$:
\begin{equation}\label{eq:function theta}
    f_{\theta}(X, x, y) = \min \{z\!>\!2^y\colon\, \fale{y'}{y}\, \ele{z'}{z}\,\theta(X, x, y', z') \}.
\end{equation}
Note that (\ref{eq:function theta}), as a definition of the graph of $f_\theta$, is a $\delz(X)$ definition. We emphasize that the first two arguments of the function $f_\theta$ are always some fixed parameters, so by iterations of $f_\theta$ on some number $y$ we mean the values $f_\theta(X, x, y)$, $f_\theta(X, x, f_\theta(X, x, y))$, and so on.

Then, the algorithm finds the Ackermann translations (as described before \cref{t2-delta0-elementarity}) of the formula $\theta$ and of the definition of $f_\theta$. This will be needed for the construction of proofs in the first-order theory $\sctheory$. The translations are readily constructed in polynomial time. 
Let us note that the first-order version of $f_\theta$ takes as its first parameter, instead of a set $X$, some number $u$ seen as a code for a (finite) set. To enhance readability, we will slightly abuse notation and denote these translations also by $\theta$ and $f_\theta$, respectively, which should not lead to any confusion.

The main part of the proof of  \eqref{eq:reflection lemma} consists of proving two claims. The first one guarantees that, provably in $\rcas+\neg\iszo+\lpc$, if the sentence $\exists X\,\exists x\,\forall y\,\exists z\,\theta(X, x, y, z)$ holds, then $\mathbf{1}$ forces that there exists a number $y$ such that for some number $u$, the function $f_{\theta}(u, y, \cdot)$ can be iterated on $y$ more than $\mathbb{I}$-many times (note that $y$ occurs also as a parameter).
Then, by the second claim, it will follow that the existence of these iterations allows to find a condition $s$ and names $V, v$ of $\tau_2$ such that $s$ forces $\forall y\,\exists z\,\theta(V, v, y, z)$, provably in $\sctheory$.

\begin{claim}\label{claim_1} $\rcas+\neg\iszo+\lpc$ proves the following sentence:
\begin{equation}\label{1-forsuje-mozna iterowac >I razy}
\exists X\exists x\,\forall y\,\exists z\,\theta(X, x, y, z)\Rightarrow\forces{\mathbf{1}}{\tau_1}{\exists y\,\exists u\,\exists z\,\exists c\,\bigl( z=f_{\theta}^{(c)}(u, y, y)\,\wedge\, c>\mathbb{I} \bigr)},
\end{equation}
and the proof can be constructed in time polynomial in $|\theta|$.
\end{claim}
\begin{claimproof}
We will describe informally a proof in $\rcas+\neg\iszo+\lpc$ of (\ref{1-forsuje-mozna iterowac >I razy}). Its main part can be obtained from a single template into which one has to substitute $\theta$ or other expressions or proofs obtained from $\theta$ in polynomial time. The only part of the proof which is not immediately seen to be constructible in time polynomial in $|\theta|$ is the proof of statement (\ref{eq:Kleene}) which will be discussed below.

Assume the antecedent of the implication (\ref{1-forsuje-mozna iterowac >I razy}) and pick $B$ and $b$ such that $\forall y\,\exists z\,\theta(B, b, y, z)$ holds. Fix an unbounded set $A=\{a_i\}_{i\in\cutzo}$ whose existence is guaranteed by the axiom $\lpc$ and \cref{cofinal set}. 
Recall from the remarks after \cref{cofinal set} that for each number $x$ there is a unique $i\in\cutzo$ such that $x\in (a_{i-1}, a_i]$, and that the statement `$x\in (a_{i-1}, a_i]$' is expressed by a $\delo(A)$ formula of $x$ and $i$, the shape of which does not depend on the specific set $A$.

To witness the existential quantifier `$\exists u$' in the consequent of the implication (\ref{1-forsuje-mozna iterowac >I razy}) we will define a total function $\hat{B}$ that maps each number $x$ to the initial segment of the set $B$ needed to compute the first $i$ iterations of the function $f_{\theta}$ on $b$, where $i$ is such that $x\in (a_{i-1}, a_i]$.
Before we give a definition of $\hat{B}$ let us make two observations. 

Firstly, for every $i\in\cutzo$, the value $f_{\theta}^{(i)}(B, b, b)$ exists. Otherwise, the set $\{x\in\nn\colon \exists z\big(z=f_{\theta}^{(x)}(B, b, b)\bigr)\}$ would be a $\sigzo$-definable cut properly contained in $\cutzo$, contradicting the assumption that $\lpc$ holds.

Secondly, let $\psi(B, A, r, b, x, i)$ be a $\delzz$ formula expressing that $x\in (a_{i-1}, a_i]$ and $r$ is the sequence of the first $i$ iterations of $f_{\theta}(B, b, \cdot)$ starting at $b$:
\begin{equation}\label{eq:psi_r_sequence}
\begin{split} 
    \psi(B, A, r, b, x, i)\defeq x\in(a_{i-1}, a_i]\,\wedge\, |r|=i+1 
     \,\wedge\,(r)_0=b \\ \wedge\,\,\fal{j}{i}((r)_{j+1}=f_{\theta}(B, b, (r)_{j})).
\end{split}
\end{equation}
By Kleene's normal form theorem for $\sigzo$ formulas (cf. \cite[Lemma 7.13]{Hirschfeldt_Book}), 
there is a $\delzz$ formula $\psi'$, a term $t(r, x_1, x_2, x_3)$ and a proof in $\rcas$ of the following sentence:
\begin{multline}\label{eq:Kleene}
\forall X, Y\,\forall r\,\forall x_1, x_2, x_3\bigl( \,\psi(X, Y, r, x_1, x_2, x_3)\\
\Leftrightarrow\,\fage{u}{t(r, x_1, x_2, x_3)}\psi'(\restr{X}{u}, Y, r, x_1, x_2, x_3)\bigr). 
\end{multline}
In our case, because of the assumption that all quantifiers in $\theta$ are bounded by a variable, we can make the same assumption about $\psi$, and thus $t$ can be found in polynomial time, say as $(r+x_1+x_2+x_3)^{|\psi|}$ (note that such a term has size polynomial in $|\psi|$). Then one can construct in polynomial time a routine proof that $t$ majorizes any value needed to evaluate $\psi$. Let us note that for a general $\delz$ formula $\psi$ it would not be possible to build $t$ in polynomial time due to the possibility of simulating exponentiation by repeated squaring in quantifier bounds.

Now we can define the following $\delo(A, B)$-functions, the first of which is constant:
\begin{align*}
   \hat{b}(x) &= b; \\
   \begin{split}
     \hat{B}(x) &= \text{the code for } \restr{B}{u},\\
     &  \text{\,\,\,\,\,\,\,\,where } u=t(r, b, x, i), \text{ with } r \text{ unique satisfying (\ref{eq:psi_r_sequence}) and } t \text{ as in (\ref{eq:Kleene})}.
     \end{split}
\end{align*}

We show that $\hat{B}$ and $\hat{b}$ witness the outer existential quantifiers `$\exists y\,\exists u$' in (\ref{1-forsuje-mozna iterowac >I razy}). 
As in the proof of \cref{tau1-is-polynomial-FI-of-SC}, let the function $d$ be defined by $d(x)=i$, where $x\in(a_{i-1}, a_{i}]$.
We define a total $\delo(A, B)$-function $w$ as follows:

\begin{equation}\label{eq:function w}
w(x) =  f_{\theta}^{d(x)}(\hat{B}(x), \hat{b}(x), \hat{b}(x)).
\end{equation}
The formula $z=f_\theta^{(v)}(u, y, y)$ is $\delzz$ and (\ref{eq:function w}) holds for every number $x$. 
Thus, we can apply \cref{Delta_zero_Los} to get a proof that $\mathbf{1}$ forces 
$w = f_{\theta}^{(d)}(\hat{B}, \hat{b}, \hat{b})$. Finally, we construct a proof of the fact that $\forces{\mathbf{1}}{\tau_1}{d>\mathbb{I}}$ in the same way as we did when proving \cref{tau1-is-polynomial-FI-of-SC}, and thus we obtain a proof of (\ref{1-forsuje-mozna iterowac >I razy}).
\end{claimproof}

\begin{claim}\label{claim_2}
$\sctheory$ proves the following sentence:    
\begin{multline}\label{eq:claim_2}
\exists y\,\exists u\,\exists z\,\exists c\,\bigl( z=f_{\theta}^{(c)}(u, y, y)\,\wedge\, c>\mathbb{I} \bigr)\,\Rightarrow\, \\ \exists \cond{s}{\tau_2}\,\exists\nam{v, V}{\tau_2}\,\bigl(\forces{s}{\tau_2}{\forall y\,\exists z\,\theta(V, v, y, z)}\bigr),
\end{multline}
and the proof can be constructed in time polynomial in $|\theta|$.
\end{claim}
\begin{claimproof}
We describe informally a proof of \eqref{eq:claim_2} in $\sctheory$. As in the case of the first claim, it should be easily seen that the proof follows one and the same template for all $\theta$, into which one has to substitute expressions and auxiliary proofs that can be obtained from $\theta$ in polynomial time.

Assume that the antecedent of the implication holds, and let $b$, $B$ and $c>\mathbb{I}$ be such that the value $z=f_{\theta}^{(c)}(B, b, b)$ exists (note that here $B$ is a number seen as a code for a finite set). Define the finite set:
\begin{equation}\label{eq:finite set-iterated f_theta}
s = \{x\in\nn\colon \el{k}{c} \bigl( x=f_{\theta}^{(k)}(B, b, b) \bigr) \}.    
\end{equation}

By (the first-order translation of) the definition of the function $f_{\theta}$ (\ref{eq:function theta}), it follows that $s$ is exponentially sparse. By our assumption, $s$ has $c$ elements and $c>\mathbb{I}$, so $s$ is a condition of $\tau_2$.  Since $b$ is the smallest element of $s$ (for $k=0$), we have that $|s\cap[0, b]|=1\in\mathbb{I}$, so
$\validname{s}{\tau_2}{b}$. Clearly, $B$ is a valid name for a second-order object.

We show that $\forces{s}{\tau_2}{\forall y\,\exists z\,\theta(B, b, y, z)}$. 
Take any condition $\{s_1< \dots <s_k\}=\lcond{s'}{s}{\tau_2}$ and a name $v$ such that $\validname{s'}{\tau_2}{v}$.
It follows that the set $s'\cap [0, v]$ 
is not a condition, so it has $j$ elements $s_1< \dots <s_{j}$ for some $j\in\mathbb{I}$. Since $\mathbb{I}$ is a cut, it also holds that $j+2\in\mathbb{I}<k$. Thus, we can take the next two elements $s_{j+1}, s_{j+2}\in s'$ and, clearly, the set $s'\cap[0, s_{j+2}]$ is not a condition. 
Hence, $s'$ forces both $s_{j+1}$ and $s_{j+2}$ to be valid names.
By the definitions of $s$ and the function $f_{\theta}$, we learn that $\fale{y}{s_{j+1}} \ele{z}{s_{j+2}}\,\theta(B, b, y, z)$.
Since $v\!<\!s_{j+1}$, we obtain that $\ele{z}{s_{j+2}}\,\theta(B, b, v, z)$. 
The last formula is $\Delta_0$, so by \cref{t2-delta0-elementarity} we get
$\forces{s'}{\tau_2}{\!\exists z\,\theta(B, b, v, z)}$. By the definition of forcing a universal formula (\ref{item:ftr/forall}), this shows that $\forces{s}{\tau_2}{\forall y\,\exists z\,\theta(B, b, y, z)}$.
\end{claimproof}

The algorithm finishes constructing the proof of (\ref{eq:reflection lemma}) as follows.
It simulates the algorithm provided by \cref{FI-translation of whole proofs} to transform the proof of (\ref{eq:claim_2}) guaranteed by Claim~$\ref{claim_2}$ into a proof in $\rcas+\neg\iszo+\lpc$ of the sentence: 
\begin{multline}\label{eq:jest iteracja-jest warunek}
\forces{\mathbf{1}}{\tau_1}{\Bigl(\exists u\, \exists y\,\exists z\,\exists c\,\bigl( z=f_{\theta}^{(c)}(u, y, y)\wedge\, c>\mathbb{I} \bigr)\,\Rightarrow}\\ \exists\cond{s}{\tau_2}\,\exists\nam{V, v}{\tau_2}\,\bigl(\forces{s}{\tau_2}{\forall y\,\exists z\,\theta(V, v, y, z)}\bigr)\Bigr).
\end{multline}
Then, in a fixed number of steps it derives \eqref{eq:reflection lemma} from (\ref{1-forsuje-mozna iterowac >I razy}) and (\ref{eq:jest iteracja-jest warunek}).  
\end{proof}

\begin{lemma}\label{cac_poly-simulation under lpc}
$\wkls+\cac$ is polynomially simulated by $\rcas+\neg\iszo+\lpc$ with respect to $\forall\pizth$ sentences.     
\end{lemma} 
\begin{proof}
Let $\delta$ be a proof in $\wkls+\cac$ of a sentence $\forall X\forall x\,\varphi(X, x)$, where $\varphi(X, x)$ is $\sigzt$.
To avoid complicating the proof, we will ignore the distinction between $\neg\varphi$ and the $\pizt$ formula equivalent to it (by a polynomial-time constructible proof).

By \cref{composition -lemma 1} and \cref{composition - lemma 2}, there exist algorithms that output proofs in $\rcas+\neg\iszo+\lpc$ of the sentences:
\begin{equation}\label{eq:final_1}
\forces{\mathbf{1}}{\tau_1}{\bigl(\forall\cond{s}{\tau_2}\,(\forces{s}{\tau_2}{\forall X\forall x\,\varphi(X, x)})\bigr)}
\end{equation}
and 
\begin{multline}\label{eq:final_2}
\exists X\exists x\,\neg\varphi(X, x)\Rightarrow\\ \forces{\mathbf{1}}{\tau_1}{(\exists\cond{s}{\tau_2}\,\exists\nam{V, v}{\tau_2}\,(\validname{s}{\tau_2}{v}\wedge\,\forces{s}{\tau_2}{\neg\varphi(V, v)}))}    
\end{multline} 
in time polynomial in $|\delta|$ and $|\varphi(X, x)|$, respectively.
Note that one can construct in polynomial time (cf. a remark in the third paragraph after \cref{definition:FI}) a proof that infers from \eqref{eq:final_1} the sentence: 
\begin{equation}\label{eq:final_4}
\forces{\mathbf{1}}{\tau_1}{\bigl(\forall\cond{s}{\tau_2}\,\forall\nam{V, v}{\tau_2}\,\bigl(\validname{s}{\tau_2}{v}\,\Rightarrow\,\forces{s}{\tau_2}{\varphi(V, v)}\bigr)\bigr)}.
\end{equation}

On the other hand, by combining \cref{no-contradiction-forced} and \cref{FI-translation of whole proofs}, there is a polynomial-time algorithm which, given as input the formula $\varphi(X, x)$, outputs a proof in $\rcas+\neg\iszo+\lpc$ of the sentence:
\begin{equation}\label{eq:final_3}
\forces{\mathbf{1}}{\tau_1}{\forall\cond{s}{\tau_2}\,\forall\nam{V, v}{\tau_2}\neg\bigl(\forces{s}{\tau_2}{\varphi(V, v)}\,\wedge\,\forces{s}{\tau_2}{\neg\varphi(V, v)}\bigr)}.
\end{equation}

Now, in polynomial time one can obtain
a proof in $\rcas+\neg\iszo+\lpc$ of the sentence $\forall X\forall x\,\varphi(X, x)$, by 
combining \eqref{eq:final_2}, \eqref{eq:final_4} and \eqref{eq:final_3} in a fixed number of inferences using the algorithms from \cref{section:FI general}.
\end{proof}

The case when $\lpc$ fails is much simpler: as we mentioned before, we have an almost trivial interpretation (in the usual non-forcing sense) of $\sctheory$ in $\rcas+\neg\iszo+\neg\lpc$. Therefore, we can define a forcing interpretation of $\wkls+\cac$ directly in $\rcas+\neg\iszo+\neg\lpc$ and prove the reflection property as stated in \cref{def:reflection}.  

\begin{lemma}\label{cac_poly-simulation under not-lpc}
$\wkls+\cac$ is polynomially simulated by $\rcas+\neg\iszo+\neg\lpc$ with respect to $\forall\pizth$ sentences.     
\end{lemma} 
\begin{proof}
We appeal to \cref{thm:Poly-FI-give-poly-simulation}.
To define a polynomial forcing interpretation $\tau_3$ of $\wkls+\cac$ in $\rcas+\neg\iszo+\neg\lpc$ one repeats the definition of $\tau_2$ in \cref{section:generic cut} replacing each occurrence of the predicate $\mathbb{I}$ with the definition of the cut $\cutzo$. The analogues of Lemmas \ref{conditions_2 - basics}-\ref{tau2-is-polynomial-FI-of-CAC} hold and are proved in the same way because $\cutzo$ satisfies the properties expressed by the axiom $\sca$.

To see that $\tau_3$ is polynomially $\forall\pizth$-reflecting we only need to adapt (in fact, simplify) the proof of \cref{composition - lemma 2}. 
So, let us work in $\rcas+\neg\iszo+\neg\lpc$ and assume that a $\exists\sigzth$ sentence $\exists X\,\exists x\,\forall y\,\exists z\,\theta(X, x, y, z)$ holds, where $\theta$ is $\delzz$. Take a set $B$ and a number $b$ that witness the existential quantifiers `$\exists X\,\exists x$'. Define a $\delz(B)$-function $f_\theta$ and a finite set $s$ just as in \eqref{eq:function theta} and \eqref{eq:finite set-iterated f_theta}, respectively. Note that with the current assumptions the definition of $s$ makes sense as well: the value $f_{\theta}^{(c)}(B, b, b)$ exists for some number $c>\cutzo$ because the set $\{x\in\nn\colon \exists z\big(z=f_{\theta}^{(x)}(B, b, b)\bigr)\}$ is a $\sigzo$-definable cut which, by $\neg\lpc$, is a proper superset of $\cutzo$. Now one can check like in the proof of \cref{composition - lemma 2} that the condition $s$ forces $\forall y\,\exists z\,\theta(V, v, y, z)$, for the names $V\defeq$ \textit{the initial segment of} $B$ \textit{needed to compute} $f_{\theta}^{(c)}(B, b, b)$ and $v\defeq b$. This clearly implies (the contraposition of) the instance of the reflection property \eqref{eq:reflection} for the $\forall\pizth$ sentence $\forall X\,\forall x\,\exists y\,\forall z\,\neg\theta(X, x, y, z)$.
\end{proof}
Finally, we are ready to prove the main theorem of the paper.

\begin{proof}[Proof of \cref{main theorem_cac}]
Let $\delta$ be a proof in $\wkls+\cac$ of a $\forall\pizth$ sentence $\sigma$. By \cref{cac_poly-simulation under lpc}, \cref{cac_poly-simulation under not-lpc} and \cref{rt-polynomial simulation}, in time polynomial in $|\delta|$ one can find proofs of $\sigma$ in, respectively, $\rcas+\neg\iszo+\lpc$, $\rcas+\neg\iszo+\neg\lpc$ and $\rca$. Combining these three proofs by means of a simple case distinction, one obtains a proof of $\sigma$ in $\rcas$ in time polynomial in $|\delta|$.
\end{proof}

We conclude with some comments on two other Ramsey-like combinatorial principles that were studied over $\rcas$ in $\cite{weak_cousins}$: the ascending-descending sequence principle $\ads$ and the cohesive Ramsey's theorem $\crt$. 
Over $\rca$ we have the following well-known chain of implications: $\rt \Rightarrow \cac \Rightarrow \ads \Rightarrow \coh \Rightarrow \crt$. 
The usual proofs of the implications from $\rt$ to $\cac$ and $\crt$, and for the one from $\cac$ to $\ads$, do not require any induction axioms, and thus they go through over $\rcas$. Since $\ads$ is a single axiom we immediately obtain the following.

\begin{corollary}\label{cor: ads has non-speedup}
$\wkls+\ads$ is polynomially simulated by $\rcas$ with respect to $\forall\pizth$ sentences.
\end{corollary}

On the other hand, it is not known whether the implication from $\ads$ or $\cac$ to $\crt$ is provable without $\iszo$. Also, as opposed to $\rt$ and $\cac$, there is no natural finite version of $\crt$ which could be used to construct a forcing interpretation analogous to the one from \cref{section:generic cut}. As of now, we have no strong evidence concerning the question whether $\rcas+\crt$ has speedup over $\rcas$. For instance, \cite{weak_cousins} left open the question about $\crt$ implying any nontrivial closure properties of the cut $\cutzo$. 
Thus, we leave the following question.

\begin{question}
Is $\rcas+\crt$ polynomially simulated by $\rcas$ with respect to $\forall\pizth$ sentences?
\end{question}

\paragraph{Acknowledgement}The author is very grateful to her PhD advisor, Leszek Ko\l{}odziejczyk, for introducing her to the topic of proof size and for many helpful comments on the paper. 

\end{document}